\documentclass[letterpaper,11pt]{amsart}

\usepackage[all]{xy}                        %

\usepackage[margin=1in]{geometry}

\CompileMatrices                            

\UseTips                                    

\input xypic
\usepackage[bookmarks=true]{hyperref}       

\usepackage{amssymb,latexsym,amsmath,amscd}
\usepackage{xspace}
\usepackage{color}

\usepackage{graphicx}

\usepackage[utf8]{inputenc}
\usepackage{fourier}
\usepackage{array}
\usepackage{makecell}

\usepackage{array}
\newcolumntype{L}[1]{>{\raggedright\let\newline\\\arraybackslash\hspace{0pt}}m{#1}}
\newcolumntype{C}[1]{>{\centering\let\newline\\\arraybackslash\hspace{0pt}}m{#1}}
\newcolumntype{R}[1]{>{\raggedleft\let\newline\\\arraybackslash\hspace{0pt}}m{#1}}

\reversemarginpar

\theoremstyle{plain}
\newtheorem{theorem}{Theorem}[section]
\newtheorem*{theorem*}{Theorem}
\newtheorem{proposition}[theorem]{Proposition}

\newtheorem{lemma}[theorem]{Lemma}

\theoremstyle{definition}
\newtheorem{definition}[theorem]{Definition}

\newtheorem{remark}[theorem]{Remark}

\newcommand{\enm}[1]{\ensuremath{#1}}          %
\newcommand{\op}[1]{\operatorname{#1}}
\newcommand{\cal}[1]{\mathcal{#1}}

\newcommand{\II}{\enm{\mathbb{I}}}
\newcommand{\JJ}{\enm{\mathbb{J}}}

\newcommand{\ZZ}{\enm{\mathbb{Z}}}

\newcommand{\PP}{\enm{\mathbb{P}}}

\newcommand{\VV}{\enm{\mathbb{V}}}

\newcommand{\EE}{\enm{\mathbb{E}}}
\newcommand{\UU}{\enm{\mathbb{U}}}

\newcommand{\Aa}{\enm{\cal{A}}}
\newcommand{\Bb}{\enm{\cal{B}}}

\newcommand{\Dd}{\enm{\cal{D}}}
\newcommand{\Ee}{\enm{\cal{E}}}
\newcommand{\Ff}{\enm{\cal{F}}}
\newcommand{\Gg}{\enm{\cal{G}}}

\newcommand{\Ii}{\enm{\cal{I}}}
\newcommand{\Jj}{\enm{\cal{J}}}

\newcommand{\Ll}{\enm{\cal{L}}}
\newcommand{\Mm}{\enm{\cal{M}}}

\newcommand{\Oo}{\enm{\cal{O}}}

\newcommand{\Ss}{\enm{\cal{S}}}

\newcommand{\Vv}{\enm{\cal{V}}}
\newcommand{\Ww}{\enm{\cal{W}}}

\renewcommand{\phi}{\varphi}
\renewcommand{\theta}{\vartheta}
\renewcommand{\epsilon}{\varepsilon}

\newcommand{\Proj}{\op{Proj}}
\newcommand{\Hom}{\op{Hom}}
\newcommand{\Ext}{\op{Ext}}

\newcommand{\End}{\op{End}}

\renewcommand{\to}[1][]{\xrightarrow{\ #1\ }}

\newcommand{\ceil}[1]{\lceil #1 \rceil}

\newcommand{\old}[1]{}

\begin{document}

\title[aCM vector bundles on projective surfaces]{aCM vector bundles on projective surfaces \\of nonnegative Kodaira dimension}

\author{E. Ballico, S. Huh and J. Pons-Llopis}

\address{Universit\`a di Trento, 38123 Povo (TN), Italy}
\email{edoardo.ballico@unitn.it}

\address{Sungkyunkwan University, Suwon 440-746, Korea}
\email{sukmoonh@skku.edu}


\address{Department of Information Engineering, Computer Science and Mathematics - University of L'Aquila, Italy}
\email{juanfrancisco.ponsllopis@univaq.it}

\keywords{polarized surfaces, arithmetically Cohen-Macaulay sheaf, wild representation type}
\thanks{The first author is partially supported by GNSAGA of INDAM (Italy) and MIUR PRIN 2015 \lq Geometria delle variet\`a algebriche\rq. The second author is supported by the National Research Foundation of Korea(NRF) grant funded by the Korea government(MSIT) (No. 2018R1C1A6004285 and No. 2016R1A5A1008055). The third author is supported by a Postdoctoral Fellowship DISIM 2017-B0010.}

\subjclass[2010]{Primary: {14F05}; Secondary: {13C14, 16G60}}

\begin{abstract}
In this paper we contribute to the construction of families of arithmetically Cohen-Macaulay (aCM) indecomposable vector bundles on a wide range of polarized surfaces $(X,\Oo_X(1))$ for $\Oo_X(1)$ an ample line bundle. In many cases, we show that for every positive integer $r$ there exists a family of indecomposable aCM vector bundles of rank $r$, depending roughly on $r$ parameters, and in particular they are of \emph{wild representation type}. We also introduce a general setting to study the complexity of a polarized variety $(X,\Oo_X(1))$ with respect to its category of aCM vector bundles. In many cases we construct indecomposable vector bundles on $X$ which are aCM for all ample line bundles on $X$.
\end{abstract}

\maketitle



\section{Introduction}
In many areas of mathematics it plays a central role to understand the \emph{complexity} of the objects one is interested in. This complexity can be measured in many different ways. For instance, in representation theory of quivers, Gabriel's theorem states that a connected quiver supports only finitely many irreducible representations, i.e., of indecomposable modules over the associated path algebra, if and only if it is of type $A$, $D$, $E$. The classification of \emph{tame} quivers as {\it Euclidean graphs}, or {\it extended Dynkin diagrams}, of type $\tilde A$, $\tilde D$, $\tilde E$ was obtained right after. Remarkably, any other quiver supports arbitrarily large families of indecomposable representations, i.e., they turn out to be of \emph{wild representation type}.

Motivated by the results, similar questions were raised to understand the category of Cohen-Macaulay modules over an arbitrary $\mathbf{k}$-algebra $R$. When $R:=\mathbf{k}[x_0,\dots,x_n]/I$ is a graded algebra finitely generated in degree one over a field $\mathbf{k}$, Cohen-Macaulay modules correspond naturally to arithmetically Cohen-Macaulay sheaves over the closed subscheme $\Proj(R)\subset\PP^n$; see \cite{kleiman-landolfi}.

\begin{definition}\label{deff}
A coherent sheaf $\Ee$ on a projective scheme $(X, \Oo_X(1))$ is called {\it arithmetically Cohen-Macaulay} (for short, aCM) if the following conditions hold:
\begin{itemize}
\item [(i)] $\Ee$ is locally Cohen-Macaulay, i.e. the stalk $\Ee_x$ has depth equal to $\dim \Oo_{X,x}$ for any point $x$ on $X$;
\item [(ii)] $H^i(\Ee(t))=0$ for all $t\in \ZZ$ and $i=1,\ldots, \dim X-1$.
\end{itemize}
\end{definition}


\noindent The forementioned correspondence allowed to use a geometrical approach to this kind of questions. A milestone on this area was due to Horrocks, stating that the only indecomposable aCM sheaf on $\PP^n$, up to twist, is $\Oo_{\PP^n}$; see \cite{horrocks:punctured}. A similar classification was obtained for a smooth quadric hypersurface $Q\subset\PP^n$: there exist, besides the structural sheaf $\Oo_Q$, only one (for $n$ even) or two (for $n$ odd) irreducible aCM sheaves, the well-studied Spinor bundles; see \cite{Knorrer}. The combined work of many mathematicians allowed to complete the list of projective schemes -of positive dimension- supporting a finite number of aCM sheaves, called the varieties of \emph{finite aCM-representation type}: they are either a projective space $\PP^n$, a smooth quadric hypersurface $X\subset \PP^n$, a cubic scroll in $\PP^4$, the Veronese surface in $\PP^5$ or a rational normal curve; see \cite{eisenbud-herzog:CM}.

The next degree of complexity is offered by the elliptic curves: in this case, vector bundles of a given rank and degree on an elliptic curve $C$ are in bijection with the points of $C$; see \cite{atiyah:elliptic}. They are called varieties of \emph{tame aCM-representation type}. In \cite{faenzi-malaspina} it was shown that smooth  quartic surface scrolls in $\PP^5$ are also tame. Notice that all the projective schemes $X\subset\PP^n$ mentioned until now are arithmetically Cohen-Macaulay, namely the coordinate ring $R:=\mathbf{k}[x_0,\dots,x_n]/I_X$ is a Cohen-Macaulay ring. Indeed, the represention type of the remaining aCM projective schemes $X\subset\PP^n$ was set in \cite{faenzi-pons}: they support arbitrarily large families of indecomposable non-isomorphic aCM sheaves. They are, therefore, of \emph{wild aCM-representation type.}

On the other hand, up to our knowledge, a broader problem has been much less studied: which are the possible dimensions of families of aCM irreducible sheaves on polarized schemes $(X,\Oo_X(1))$, where the only requirement for the line bundle $\Oo_X(1)$ is to be ample. With this setting it is proved in \cite{casnati-special-ulrich} and \cite{casnati-special-ulrich-2} that polarized surfaces $(S,\Oo_S(1))$ such that $p_g=0$, $q=0$ or $1$, and $\Oo_S(1)$ is very ample with $h^1(\Oo_S(1))=0$ are of wild representation type. Indeed, the aCM vector bundles witnessing wilderness own a special property: they have the maximal permitted number of global sections, namely they are the so-called \emph{Ulrich vector bundles}. Again for $\Oo_X(1)$ very ample, it is proved in \cite{pons} that for polarized varieties $(X,\Oo_X(1))$ of dimension at least two, the embedding given by $\Oo_X(l)$ with $l\geq 3$ is of wild representation type under some mild assumptions on $\Oo_X(1)$.

The goal of the present paper is to contribute to this set of problems: we are constructing families of aCM vector bundles on a large range of polarized integral surfaces $(X, \Oo_X(1))$. In the following Theorem we summarize the results obtained:

\begin{theorem}\label{thm1.1}
Let $X$ be an integral projective surface with a fixed ample line bundle $\Oo_X(1)$ listed below. Then for each integer $r\ge 2$ there exists an $b_X(r)$-dimensional irreducible family $\{\Ee_{\alpha}\}_{\alpha \in \Gamma}$ of indecomposable aCM vector bundles of rank $r$ on $X$ such that for each $\alpha \in \Gamma$ there are only finitely many $\beta\in \Gamma$ with $\Ee_{\alpha} \cong \Ee_{\beta}$.
$$
\begin{tabular}{|C{0.5cm}|C{11cm}|C{0.9cm}|}
  \hline
no. &   \text{$X$}   & \text{$b_X(r)$}\\
  \hline
  \hline
$1$ & \makecell{$\pi : X\rightarrow Y$ a birational morphism with $\omega_Y \cong \Oo_Y$ and $q(Y)=0$\\ such that $\pi^{-1}(Y_{\mathrm{sing}})\cong Y_{\mathrm{sing}}$} & $2r$ \\[0.07cm]
\hline
$2$ &\makecell{$\omega_X \ncong \Oo_X$ locally free with $h^0(\omega_X)=0$ and $h^0(\omega_X^{\otimes 2})=1$, and $q(X)=0$} & $2\ceil{\frac{r}{2}}$ \\[0.07cm]
\hline
$3$ &\makecell{smooth and $q(X)=1$ with $\omega_X^\vee \otimes \Oo_X(1)$ trivial or ample}  & $1$ \\[0.07cm]
\hline
$4$ & \makecell{$\pi : X \rightarrow Y$ a birational morphism with an abelian surface $Y$\\ and $\omega_X^\vee \otimes \Oo_X(1)$ trivial or ample } & $r+1$ \\[0.07cm]
\hline
$5$ & \makecell{$\pi : X \rightarrow Y$ a birational morphism with a hyperelliptic surface $Y$} & $1$\\[0.07cm]
\hline
$6$ & \makecell{$\omega_X \cong \Oo_X(1)$ with $h^1(\omega_X^{\otimes n})=0$ for all $n\in \ZZ$ and $p_g\ge 3$} & $r$\\[0.07cm]
  \hline
\end{tabular}$$
\end{theorem}

Theorem \ref{thm1.1} shows that the projective surfaces of Kodaira dimension zero, possibly with singularity, are of wild representation type, except the case of hyperelliptic surfaces. G.~Casnati proved in \cite{casnati-special-ulrich-2} that hyperelliptic surfaces are of wild representation type with respect to a very ample polarization. Note that we do not assume in Theorem \ref{thm1.1} that $X$ is minimal or $\Oo_X(1)$ is very ample, while the result in \cite{casnati-special-ulrich-2} is more powerful in a sense that it gives wildness with respect to Ulrich vector bundles. 

The strategy for Theorem \ref{thm1.1} is two-fold. One is to consider zero-dimensional subschemes of length equal to the second Chern class of the aCM vector bundles in consideration, from which we construct aCM vector bundles of arbitrary rank by a series of extensions. The cases no. $1$, $2$ and $6$ are handled by this method respectively in Theorem \ref{k3}, Theorem \ref{enriquesmain} and Theorem \ref{thm22}; in case no. $6$, for the construction of a family of aCM vector bundles of rank $r$ even, it is enough to suppose that $p_g\ge 2$. The second strategy is to consider a family of aCM line bundles, parametrized by a non-empty open Zariski subset of $\mathrm{Pic}^0(X)$, from which we construct aCM vector bundles of arbitrary rank by iterated extensions. the cases no. $3$, $4$ and $5$ are handled by this method respectively in Proposition \ref{c1}, Theorem \ref{cc1} and Proposition \ref{hyperelliptic}. 

Based on the results in Theorem \ref{thm1.1} we introduce a set-up to measure the complexity of a polarized variety $(X, \Oo_X(1))$. Define
$$a_{X, \Oo_X(1)}(r):=\mathrm{sup}_{\Gamma} \Bigg\{ \dim \Gamma~ \Bigg |   ~\makecell{\Gamma \text{ runs over the parameter spaces of indecomposable}\\ \text{ aCM vector bundles of rank $r$ on $X$}} \Bigg  \}$$
with the convention that $a_{X, \Oo_X(1)}(r)=-\infty$ if there is no indecomposable aCM vector bundle of rank $r$. Then we have $a_{X, \Oo_X(1)}(r) \ge b_X(r)$ for the surfaces listed in Theorem \ref{thm1.1}. We also define
\begin{align*}
a_X(r)&:=\mathrm{sup}\left \{ a_{X, \Oo_X(1)}(r) ~|~ \Oo_X(1) \text{ ample}\right \}~,~a_X'(r):=\mathrm{inf}\left \{ a_{X, \Oo_X(1)}(r) ~|~ \Oo_X(1) \text{ ample}\right \}.
\end{align*}
In many construction of aCM vector bundles, the polarization is assumed to be very ample, in which case we give similar definitions for $a_X(r)$ and $a_X'(r)$, if we consider only very ample polarizations in their definitions. Then we may raise several questions.
\begin{itemize}
\item For a given $X$, what can be said about the following limits?
$$\limsup _{r\to \infty} a_X(r)~,~ \limsup _{r\to \infty} a_X'(r)~,~\liminf _{r\to \infty} a_X(r)~\text{ and }~\liminf _{r\to \infty} a_X'(r)$$
\item What can be said about following suprema
$$\mathrm{sup}_X {\left \{a_X(r)\right\}} \text{ and } \mathrm{sup}_X\left \{a_X'(r)\right\},$$
where $X$ runs over all smooth projective varieties, all varieties with a prescribed Kodaira dimension or all varieties in a prescribed interesting class, e.g. K3 surfaces?
\end{itemize}
In those questions with $(X, \Oo_X(1)$ polarized surfaces, we may allow singular surfaces, but locally CM, e.g. normal or singularity with embedded dimension at most three, so that we may consider non-locally free aCM sheaves. We do not know if we may obtain bigger dimensional families of indecomposable aCM sheaves by considering non-locally free aCM sheaves.

For higher dimensional smooth varieties we prove the following result.
\begin{theorem}\label{cc1}
Let $X$ be a smooth projective variety of dimension $n\ge 2$, birational to an abelian variety and fix an ample line bundle $\Oo_X(1)$ with $\omega_X^\vee \otimes \Oo_X(1)$ ample. Then $X$ is wild with respect to $\Oo_X(1)$ and
\[
a_{X, \Oo_X(1)}(r) \ge (n-1)r+1.
\qedhere \]
\end{theorem}
For the proof of Theorem \ref{cc1} we use in an essential way a construction by S. Mukai of vector bundles on abelian varieties in \cite{muk}, a generic vanishing for smooth varieties with maximal Albanese dimension in \cite{gl1,gl2} and results on the local Hilbert schemes in \cite{bi, g}.

\begin{remark}
In cases no. $1$, $2$ and $6$ of Theorem \ref{thm1.1} the indecomposable vector bundles that we construct are aCM for any ample line bundles on $X$. On the other hand, in cases no. $3$, $4$ and $5$ of Theorem \ref{thm1.1} and Theorem \ref{cc1} the indecomposable vector bundles that we construct are aCM for every ample line bundles $\Oo_X(1)$ with $\omega _X^\vee \otimes \Oo_X(1)$ ample.
\end{remark}

Recall from Theorem \ref{thm1.1} that we obtain irreducible families of indecomposable aCM vector bundles of rank $r$ on several projective surfaces, whose dimensions are at most linear polynomials in $r$. Nonetheless, we may not expect that $a_{X, \Oo_X(1)}(r)$ is linear in $r$ for any projective surface. Indeed, Remark \ref{cc2} shows that for $X$ as in Theorem \ref{cc2} with $n\ge 3$ we get a lower bound for $a_{X,\Oo_X(1)}(r)$ greater than linear, but less than quadratic, in $r$.

\begin{remark}\label{cc2}
Let $X$ be as in Theorem \ref{cc1}. For $n\ge 3$ and $r\gg 0$, $\UU_r$ is reducible and there are positive constants $\alpha _n$ and $\beta _n$ such that
$$\alpha _n r^{2-2/n} \le \dim B_f[r] \le \beta _n r^{2-2/n}$$
by \cite{bi} and \cite[page 6]{g}. Then from $\dim B_f[r] \le a_{X, \Oo_X(1)}(r)$ we get
\[
\liminf _{r\to \infty} ~\frac{a_r(X,\Oo _X(1))}{r^{2-2/n}} >0.
\]
\end{remark}

\noindent On the other hand, in Section \ref{gege} we suggest examples of smooth surfaces of general type with at least a quadratic lower bound for $a_{X, \Oo_X(1)}(r)$.

We would like thank C.~Ciliberto for suggesting this problem.







\section{K3-like surfaces}

In this section we assume that $X$ is integral with $\omega _X\cong \Oo _X$ and $q(X)=0$. Set $\tilde{g}:=h^0(\Oo_X(1))$; if $X$ is a K3 surface, then we have $2\tilde{g}-4=d$ and $g:=\tilde{g}-1$ is called the genus of $X$.

\begin{proposition}\label{k1}
For each $r\in \ZZ$ with $2\le r \le \tilde{g}$, there exists an indecomposable aCM vector bundle $\Ee$ of rank $r$ on $X$ with $\det (\Ee )\cong \Oo _X$ and $c_2(\Ee )=r$
\end{proposition}

\begin{proof}
Take a general set of points $S\subset X_{\mathrm{reg}}$ with $|S|= r$ and let $\Ee $ be a general sheaf fitting into the following exact sequence
\begin{equation}\label{eqk1}
0 \to \Oo _X^{\oplus (r-1)} \stackrel{j}{\to} \Ee \to \Ii _{S,X} \to 0.
\end{equation}
Note that $\mathrm{ext}_X^1(\Ii_{S,X}, \Oo_X)=h^1(\Ii_{S,X})=r-1$ and the sheaf $\mathrm{Im}(j)$ is the image of the evaluation map $H^0(\Ee)\otimes \Oo_X \rightarrow \Ee$. By generality of the extension (\ref{eqk1}) we may choose a basis $\{\epsilon_1, \ldots, \epsilon_{r-1}\} $ of $\mathrm{Ext}_X^1(\Ii_{S,X}, \Oo_X)$ inducing (\ref{eqk1}). In particular, $\Ee$ has no trivial factor. In case $r=2$, the sheaf $\Ee$ is locally free from the Cayley-Bacharach condition. For $r>2$, the vector bundle $\Ff \oplus \Oo_X^{\oplus (r-2)}$ with a vector bundle $\Ff$ of rank two fitting into (\ref{eqk1}) as the middle term for $r=2$, also fits into (\ref{eqk1}) as the middle term. Since local freeness is an open condition, the sheaf $\Ee$ is locally free.

Assume $\Ee \cong \Ff _1\oplus \Ff _2$ with $\mathrm{rank}(\Ff _1 )=s$ and $0 < s<r$. For each $i\in \{1,2\}$, let $\Gg _i\subseteq \Ff _i$ be the image of the evaluation map $H^0(\Ff _i)\otimes \Oo _X\rightarrow \Ff _i$ with $s_i:=\mathrm{rank}(\Gg _i)$. Then we get $\Gg _1\oplus \Gg _2 \cong \Oo _X^{\oplus (r-1)}$. In particular, each $\Gg _i$ is trivial and $s_1\in \{s, s-1\}$. Note that $(\Ff _1/\Gg _1)\oplus (\Ff _2/\Gg _2)\cong \Ii _{S,X}$ has no torsion. If $s_1=s$, then we get $\Ff_1/\Gg_1 \cong 0$, i.e. $\Ff_1 \cong \Oo_X^{\oplus s}$, which is impossible since $\Ee$ has no trivial factor. If $s_1=s-1$, then we would get a contradiction similarly from $\Ff_2 \cong \Oo_X^{\oplus (r-s)}$. Thus $\Ee$ is indecomposable.

Then it remains to show that $\Ee$ is aCM. Since $h^0(\Oo_S)\le h^0(\Oo _X(1))$ and $S$ is general, we have $h^1(\Ii _{S,X}(t)) =0$ for all $t>0$. Now $\{\epsilon _1,\dots ,\epsilon _{r-1}\}$ is a basis for $\mathrm{Ext}_X^1(\Ii_{S,X}, \Oo_X)$ and so it induces an isomorphism $H^1(\Ii_{S,X}) \rightarrow H^2(\Oo_X^{\oplus (r-1)})$. Thus we have $h^0(\Ee(t))=0$ for all $t\ge 0$. For any $\lambda \in \mathbf{k}$, let $\Ee _{\lambda}$ denote the middle term of the extension corresponding to $(\epsilon _1,\lambda \epsilon _2,\dots ,\lambda \epsilon _{r-1})$; we have $\Ee _\lambda \cong \Ee$ for $\lambda \ne 0$ and $\Ee _0 \cong \Gg \oplus \Oo _X^{\oplus (r-2)}$ with $\Gg$ induced by the extension $\epsilon_1$. As above we see that $h^1(\Gg (t)) =0$ for all $t\ge 0$. Since $\Gg$ is locally free from the Cayley-Bacharach condition and generality of $\epsilon_1$, we use Serre's duality to obtain $h^1(\Gg (t)) =h^1(\Gg (-t)) =0$ for $t<0$. Thus $\Ee_0$ is aCM. Now using the semicontinuity theorem for cohomology, we obtain $h^1(\Ee(t))=0$ because $\Ee _\lambda \cong \Ee$.
\end{proof}

\begin{remark}\label{k0}
Consider the exact sequence (\ref{eqk1}) with $r=2$. Since $\mathrm{ext}_X^1(\Ii_{S,X}, \Oo_X)=h^1(\Ii_{S,X})=1$, there exists a unique nontrivial extension of $\Ii_{S,X}$ by $\Oo_X$; denote its middle term by $\Gg_S$. Since the Cayley-Bacharach condition is satisfied, the sheaf $\Gg_S$ is an aCM vector bundle of rank two on $X$.
\end{remark}

\begin{theorem}\label{k2}
For each integer $2\le r\le \tilde{g}$, there exists a $2r$-dimensional family $\{\Ee _\alpha\}_{\alpha \in \Gamma}$ of indecomposable aCM vector bundles of rank $r$ on $X$ with $\det (\Ee _\alpha )\cong \Oo _X$ and $c_2(\Ee _\alpha )=r$ such that for each $\alpha \in \Gamma$ there are only finitely many $\beta \in \Gamma$ with $\Ee _\beta \cong \Ee
_\alpha$.
\end{theorem}

\begin{proof}
For any subset $S\subset X_{\mathrm{reg}}$ with $|S|=r$, define $\mathbb{E}'(S)$ to be the subset of $\mathbb{E}(S):=\mathrm{Ext}_X^1(\Ii_{S,X}, \Oo_X^{\oplus (r-1)})$, consisting of all extensions whose corresponding middle terms are aCM and indecomposable vector bundles. By Proposition \ref{k1}, $\EE '(S)$ is a non-empty open subset of $\mathbb{E}(S)$ and each $[\Ee] \in \EE '(S)$ has trivial determinant with $c_2(\Ee ) =r$.

Letting $\UU:=\{S\subset X_{\mathrm{reg}}~|~|S|=r\}$, there is a vector bundle $\Vv$ of rank $(r-1)^2$ on $\UU$ with $\EE(S)$ as its fibre over $S\in \UU$, since $\mathrm{ext}_X^1(\Ii _{S,X},\Oo _X^{\oplus (r-1)}) = (r-1)^2$ for all $S\in \UU$. Then there is a non-empty open subset $\Vv' \subset \Vv$ with $\Vv ' _{|S} = \EE '(S)$ for a general $S\in \UU$. Thus there exists an irreducible variety $\Gamma \subset \Vv'$ such that the restriction of the map $\Vv \rightarrow \UU$ to $\Gamma$ is quasi-finite and dominant. In particular, we have $\dim \Gamma = \dim \UU =2r$.

For $[\Ee] \in \EE '(S) $ we have $h^0(\Ee )=r-1$ and the cokernel of the evaluation map $H^0(\Ee )\otimes \Oo _X\rightarrow \Ee$ is isomorphic to $\Ii _{S,X}$. Thus for $[\Ee]\in \EE'(S)$ and $[\Ff]\in \EE'(S')$ with $S\ne S' \in \UU$, we have $\Ee \ncong \Ee'$. Since the map $\Gamma \rightarrow \UU$ is quasi-finite, the variety $\Gamma$ satisfies the requirements for the assertion.
\end{proof}



\begin{theorem}\label{k3}
For each integer $r\ge 2$, there exists an $2r$-dimensional family $\{\Ee _\alpha\}_{\alpha \in \Gamma}$ of indecomposable aCM vector bundles of rank $r$ on $X$ with $\det (\Ee _\alpha )\cong \Oo _X$ and $c_2(\Ee _\alpha )=r$ such that for each $\alpha \in \Gamma$ there are only finitely many $\beta \in \Gamma$ with $\Ee _\beta \cong \Ee_\alpha$.
\end{theorem}

For the proof of Theorem \ref{k3} we collect numerous technical results below. We fix subsets $S_0, \dots, S_m\subset X_{\mathrm{reg}}$ with $|S_0|=3$ and $|S_i|=2$ for all $1\le i \le m$ such that $S_i \cap S_j=\emptyset$ for any $i\ne j$.

Set $\II (S_1):= \{\Ii _{S_1,X}\}$ and define $\II (S_1, \dots, S_i)$ for $i\ge 2$ inductively to be the set of all sheaves admitting an extension of $\Ii _{S_i,X}$ by an element in $\II (S_1,\dots ,S_{i-1})$. Thus for each $i\ge 2$ each sheaf $\Jj \in \II (S_1,\dots ,S_i)$ admits the following exact sequence for some $\Jj' \in \II (S_1,\dots ,S_{i-1})$
\begin{equation}\label{yyy}
0 \to \Jj' \to \Jj \to \Ii _{S_i,X}\to 0.
\end{equation}
For a subset $N=\{i_1, \ldots, i_k\} \subset \{1,\ldots, i\}$ with $i_1<\ldots<i_k$, we denote $\II(S_{i_1}, \ldots, S_{i_k})$ by $\II(S_j; j\in N)$.

Set $\JJ(\emptyset;S_0):= \{\Ii_{S_0,X}\}$ and define $\JJ(S_1, \dots, S_i; S_0)$ to be the set of all isomorphism classes of extensions of $\Ii_{S_0, X}$ by an element in $\II(S_1, \ldots, S_i)$. Similarly we define $\JJ (S_j; j\in N; S_0)$.

\begin{lemma}\label{n2}
Each sheaf $\Jj \in \II(S_1, \ldots, S_i)$ admits an exact sequence
\begin{equation}\label{ttt}
0\to \Jj  \stackrel{\iota}{\to} \Jj^{\vee\vee}\cong \Oo_X^{\oplus i} \to\Oo_{S_1\cup \dots \cup S_i} \to 0,
\end{equation}
where the map $\iota$ is the double dual map. In particular, we have $h^0(\Jj)=0$ and $h^1(\Jj)=h^2(\Jj)=i$.
\end{lemma}

\begin{proof}
The assertion is clear for $i=1$, i.e. $\Jj=\Ii_{S_1,X}$. Assume $i\ge 2$ and consider an exact sequence (\ref{yyy}) with $\Jj' \in \II(S_1, \ldots, S_{i-1})$. By inductive hypothesis, the assertion holds for $\Jj'$ and $\Ii_{S_i,X}$ and we get the following commutative diagram:
$$\begin{array}{ccccccc}
&0 & &0& &0 & \\
& \downarrow& &\downarrow & &\downarrow &\\
0\to &  \Jj'        &\to  &\Jj  &  \to  &\Ii_{S_i,X}  &\to 0 \\
& \downarrow& &\downarrow & &\downarrow & \\
0 \to &\Oo_X^{\oplus (i-1)}&\to & \Jj^{\vee\vee} &\to & \Oo_X& \to 0\\
&\downarrow& &\downarrow & &\downarrow &\\
0 \to &\Oo_{S_1\cup \dots \cup S_{i-1}} & \to &  \Jj^{\vee\vee}/\Jj &\to & \Oo_{S_i} &  \to 0\\
&\downarrow& &\downarrow & & \downarrow& \\
&0 &          &0& & 0& \\
\end{array}$$
Since $\mathrm{ext}_X^1(\Oo_X, \Oo_X)=h^1(\Oo_X)=0$, we get $\Jj^{\vee\vee} \cong \Oo_X^{\oplus i}$ from the second horizontal sequence. From the third horizontal sequence, we get $\Jj^{\vee\vee}/\Jj \cong \Oo_{S_1 \cup \dots \cup S_i}$, because $S_i$'s are disjoint to each other. Then we get the exact sequence (\ref{ttt}). The vanishing $H^0(\Jj)=0$ can be obtained by induction on $i$ and $h^1(\Jj)=h^2(\Jj)=i$ can be obtained from (\ref{ttt}).
\end{proof}

\begin{remark}\label{en2}
By the same argument in the proof of Lemma \ref{n2}, we have an exact sequence
$$0\to \tilde{\Jj} \to \tilde{\Jj}^{\vee\vee} \cong \Oo_X^{\oplus (i+1)} \to \Oo_{S_0\cup S_1 \cup \dots \cup S_i} \to 0,$$
for $\tilde{\Jj}\in \JJ(S_1, \dots, S_i; S_0)$. This gives $h^0(\tilde{\Jj})=0$, $h^1(\tilde{\Jj})=i+2$ and $h^2(\tilde{\Jj})=i+1$.
\end{remark}

\begin{lemma}\label{k00}
For a sheaf $\Jj \in \II (S_1,\dots ,S_i)$ and any finite subset $A\subset X$,
\begin{itemize}
\item [(i)] if $A\nsubseteq S_j$ for all $1\le j \le i$, then we have $\Hom_X (\Jj, \Ii _{A,X})=0$;
\item [(ii)] if $A\nsupseteq S_j$ for some $1\le j \le i$, then we have $\Hom_X (\Ii _{A,X},\Jj )=0$.
\end{itemize}
\end{lemma}

\begin{proof}
We only prove part (i), because part (ii) can be obtained similarly. Let us use induction on $i$; the case $i=1$ is true, because $A\nsubseteq S_1$ is equivalent to $\Hom_X (\Ii _{S_1,X},\Ii _{A,X})=0$. Now assume $i\ge 2$ and consider the sequence (\ref{yyy}) with $\Jj \in \II(S_1, \dots, S_{i-1})$. Since $\Hom_X (\Ii _{S_i,X},\Ii _{A,X})=0$, any map $f\in \Hom_X (\Jj, \Ii _{A,X})$ is uniquely determined by $f'\in \Hom_X (\Jj' ,\Ii _{A,X})$. The inductive assumption gives $f'=0$ and so we have $f=0$.
\end{proof}

\begin{lemma}\label{n4}
We have $\mathrm{ext}_X^1( \Ii_{S_{i+1}, X}, \Jj)=2i$ for $\Jj \in \II(S_1, \dots, S_i)$.
\end{lemma}
\begin{proof}
Let $S:=S_1 \cup \dots \cup S_i$ and apply the functor $\Hom_X(\Ii_{S_{i+1}, X}, -)$ to the sequence (\ref{ttt}) to obtain
\begin{align*}
0& \to \Hom_X( \Ii_{S_{i+1},X}, \Jj ) \to \Hom_X(\Ii_{S_{i+1}, X}, \Oo_X^{\oplus i}) \to \Hom_X(\Ii_{S_{i+1}, X}, \Oo_S)\\
&\to \Ext_X^1(\Ii_{S_{i+1},X}, \Jj ) \to \Ext_X^1(\Ii_{S_{i+1}, X}, \Oo_X^{\oplus i}) \to \Ext_X^1(\Ii_{S_{i+1}, X}, \Oo_S).
\end{align*}
Here, we have $\mathrm{hom}_X(\Ii_{S_{i+1}, X}, \Oo_X^{\oplus i})=i=\mathrm{ext}_X^1(\Ii_{S_{i+1}, X}, \Oo_X^{\oplus i})$. We also get $\mathrm{hom}_X(\Ii_{S_{i+1}, X}, \Oo_S)=2i$, because $S$ is disjoint from $S_{i+1}$. Now apply the functor $\Hom_X(-, \Oo_S)$ to the standard exact sequence for $S_{i+1}\subset X$ to obtain
$$\Ext_X^1(\Oo_X, \Oo_S) \to \Ext_X^1(\Ii_{S_{i+1}, X}, \Oo_S) \to \Ext_X^2 (\Oo_{S_{i+1}}, \Oo_S) .$$
Here, we have $\mathrm{ext}_X^1(\Oo_X, \Oo_S)=h^1(\Oo_S)=0$ and $\mathrm{ext}_X^2(\Oo_{S_{i+1}}, \Oo_S)=0$. In particular, we get $\mathrm{ext}_X^1(\Ii_{S_{i+1}, X}, \Oo_S)=0$. Finally, apply the functor $\Hom_X(\Ii_{S_{i+1}, X}, -)$ to the sequence (\ref{yyy}) to have
$$\Hom_X(\Ii_{S_{i+1}, X}, \Jj' ) \to \Hom_X(\Ii_{S_{i+1}, X}, \Jj) \to \Hom_X(\Ii_{S_{i+1}, X}, \Ii_{S_i, X}). $$
Since $S_i \cap S_{i+1}=\emptyset$, we get $\mathrm{hom}_X(\Ii_{S_{i+1}, X}, \Ii_{S_i, X})=0$. By inductive hypothesis, we get $\mathrm{hom}_X(\Ii_{S_{i+1}, X}, \Jj')=0$. Thus we have $\mathrm{hom}_X(\Ii_{S_{i+1}, X}, \Jj)=0$ and we get the assertion.
\end{proof}

\begin{remark}\label{ess}
Similarly as in the proof of Lemma \ref{n4}, we see that $\mathrm{ext}_X^1(\Ii_{S_0, X}, \Jj)=3i$ for any $\Jj \in \II(S_1, \dots, S_i)$. In particular, there exists a non-trivial extension
$$0\to \Jj \to \tilde{\Jj} \to \Ii_{S_0, X} \to 0.$$
In this case, we have $\mathrm{ext}_X^1(\Ii_{S_0, X}, \Oo_X^{\oplus i})=2i$ and the other numeric data in the proof of Lemma \ref{n4} are all same.
\end{remark}


\begin{lemma}\label{n3}
For each $i\ge 1$, there exists an indecomposable sheaf $\Jj \in \II (S_1,\dots ,S_i)$.
\end{lemma}

\begin{proof}
Since $\Ii _{S_1,X}$ has rank one and $X$ is an integral variety, $\Ii _{S_1,X}$ is indecomposable. Thus we may assume $i\ge 2$. Note that each $\Ii_{S_j, X}$ has the same Hilbert polynomial with respect to any polarization $\Oo_X(1)$. Thus any sheaf in $\II (S_1, \ldots, S_i)$ is strictly semistable with $\oplus_{j=1}^i \Ii_{S_j, X}$ as its Jordan-H\"older grading. Let $\Jj$ be a general sheaf fitting into an exact sequence
\begin{equation}\label{eqan11}
0\to \oplus_{j=1}^{i-1} \Ii_{S_j, X} \stackrel{f}{\to} \Jj \stackrel{g}{\to} \Ii_{S_i, X} \to 0
\end{equation}
and assume that $\Jj$ is decomposable, say $\Jj \cong \Aa_1\oplus \cdots \oplus \Aa_h$ with $h\ge 2$ and each $\Aa_j$ indecomposable. Since $\Jj$ is strictly semistable with $gr(\Jj)\cong \oplus_{j=1}^i \Ii_{S_j, X}$, there is a subset $N_j \subset \{1,\dots, i\}$ for each $j\in \{1,\ldots, h\}$ such that $gr(\Aa_j) \cong \oplus_{k\in N_j}\Ii_{S_k, X}$. Note that $\{N_j|1\le j \le h\}$ forms a partition of $\{1,\ldots, i\}$ with each $N_j$ non-empty.

Assume first that $|N_j|=1$ for all $j$. Then we have $\Jj \cong \oplus_{j=1}^i \Ii_{S_j, X}$. Since we have $\Hom_X(\Ii_{S_i, X}, \Ii_{S_j, X})=0$ for all $j<i$ and $\Hom_X(\Ii_{S_i, X}, \Ii_{S_i, X}) \cong \mathbf{k}$, we get that the sequence (\ref{eqan11}) splits, contradicting Lemma \ref{n4}.

Now without loss of generality, assume $e:=|N_1|\ge 2$. If $i\notin N_1$, then by permuting the first $i-1$ indices of $S_j$'s we may assume $\Aa_1 \in \II(S_1, \ldots, S_e)$. Then by Lemma \ref{k00} we have $\mathrm{hom}_X(\Ii_{S_j, X}, \Aa_1)=\mathrm{hom}_X(\Aa_1, \Ii_{S_i, X})=0$ for all $j \ge e+1$. Thus $f$ induces an isomorphism $f' : \Aa_1 \rightarrow \oplus_{j=1}^e \Ii_{S_j, X}$, contradicting the assumption $e\ge 2$ and the indecomposability of $\Aa_1$. If $i\in N_1$, then by permuting the first $i-1$ indices of $S_j$'s we may assume $\Aa_1 \in \II(S_{i-e+1}, \dots, S_i)$. From the case when $i\notin N_1$ we may also assume $|N_j|=1$ for all $j>1$, and this implies $\Jj \cong \Aa_1\oplus (\oplus_{j=1}^{i-e} \Ii_{S_j, X})$. Then by Lemma \ref{k00} we have $\Hom_X(\Ii_{S_j, X}, \Aa_1)=0$ for all $j\le i-e$. In particular, the extension class $\epsilon = (\epsilon _1,\dots ,\epsilon _{i-1})$ corresponding to (\ref{eqan11}) with $\epsilon _j\in \Ext_X^1(\Ii _{S_i,X},\Ii _{S_j,X})$ satisfies $\epsilon _j=0$ for all $j\le i-e$, contradicting Lemma \ref{n4} and the generality of $\epsilon$.
\end{proof}

\begin{remark}\label{ees1}
As in the proof of Lemma \ref{n3}, let us consider a general sheaf $\tilde{\Jj}$ fitting into an exact sequence
\begin{equation}\label{ttr1}
0\to \oplus_{j=1}^i \Ii_{S_j, X} \to \tilde{\Jj} \to \Ii_{S_0, X} \to 0.
\end{equation}
By Remark \ref{ess} the extension (\ref{ttr1}) is non-trivial. Here, $\tilde{\Jj} \in \JJ(S_1, \dots, S_i; S_0)$ and the sequence (\ref{ttr1}) is the Harder-Narasimhan filtration of $\tilde{\Jj}$. Assume that $\tilde{\Jj}$ is decomposable, say $\tilde{\Jj} \cong \tilde{\Aa_1} \oplus \cdots \oplus \tilde{\Aa_h}$. Note that the HN filtration of $\tilde{\Jj}$ is obtained from the ones of each $\tilde{\Aa_i}$. In particular, as in the proof of Lemma \ref{n3}, we have a partition $\{ N_j|1\le j \le h \}$ of $\{0,1,\cdots, i\}$ such that $\tilde{\Aa_j}\in \II(S_k ; k\in N_j)$ if $0\notin N_j$, and $\tilde{\Aa_j}\in \JJ(S_k; k\in N_j\setminus \{0\}; S_0)$. Then by the same argument in the proof of Lemma \ref{n3}, we get a contradiction. Thus we get an indecomposable sheaf in $\JJ(S_1, \dots, S_i; S_0)$.
\end{remark}

\begin{lemma}\label{n5}
For each integer $i\ge 1$, the set $\II (S_1,\dots , S_i)$ is parametrized by an affine space $T(S_1,\dots ,S_i)$, not necessarily finite-to-one, equipped with the universal sheaf, i.e. a sheaf $\Ss(S_1, \dots, S_i)$ on $T(S_1,\dots ,S_i)\times X$ such that the fiber of $\Ss(S_1, \dots, S_i)$ over $\{\Jj\}\times X$ with $\Jj\in \II(S_1, \dots, S_i)$ is the sheaf $\Jj$ on $X$.
\end{lemma}

\begin{proof}
For $i=1$ we may take as $T(S_1)$ just a single point set, because $\II (S_1) =\{\Ii_{S_1,X}\}$. Assume that there exists an affine space $T(S_1,\dots ,S_{i-1})$ and a sheaf $\Ss(S_1, \dots, S_{i-1})$ with prescribed property for $i\ge 2$. We set
\begin{align*}
T(S_1, \dots, S_i)&:=\mathcal{E}xt_{p_1}^1(\Ss(S_1, \dots, S_{i-1}), p_2^*\Ii_{S_i, X})\\
&=R^i({p_1}_*\mathcal{H}om_{T(S_1, \dots, S_{i-1})\times X}(\Ss(S_1, \dots, S_{i-1}), -))(p_2^*\Ii_{S_i, X})
\end{align*}
to be the relative $\mathcal{E}xt^1_{p_1}$-sheaf, where $p_j$ is the projection from $T(S_1,\dots, S_{i-1})\times X$ to its $j$-th factor; see \cite[Proposition 3.1]{Lange}. By Lemma \ref{n4} we have $\mathrm{ext}_X^1(\Jj' ,\Ii _{S_i,X}) =2i-2$ for each $\Jj' \in T(S_1,\dots ,S_{i-1})$. This implies that $T(S_1, \dots, S_i)$ is a vector bundle of rank $2i-2$ over $T(S_1,\dots ,S_{i-1})$ and so it is an affine space parametrizing $\II(S_1, \dots, S_i)$ as required. We may also take as $\Ss(S_1, \dots, S_i)$ the universal extension on $T(S_1, \dots, S_i)\times X$ as in \cite[Corollary 3.4]{Lange}.
\end{proof}

\begin{remark}\label{rte}
Following the same argument in the proof of Lemma \ref{n5}, we can obtain an affine space $\tilde{T}(S_1, \ldots, S_i; S_0)$ parametrizing $\JJ(S_1, \ldots, S_i)$ equipped with the universal sheaf $\tilde{\Ss}(S_1, \ldots, S_i; S_0)$.
\end{remark}

\begin{proof}[Proof of Theorem \ref{k3}:] Assume that $r$ is even and set $m:=r/2$. Fix subsets $S_1,\dots ,S_m\subset X_{\mathrm{reg}}$ such that $|S_i|=2$ for all $i$ and $S_i\cap S_j=\emptyset$ for all $i\ne j$. By Lemma \ref{n3} there exists an indecomposable sheaf $\Jj \in \II (S_1,\dots ,S_m)$, for which we consider a general sheaf $\Ee$ fitting into the following exact sequence:
\begin{equation}\label{eqn2}
0\to \Oo _X^{\oplus m} \stackrel{f}{\to} \Ee \to \Jj \to 0.
\end{equation}
Note that $\Ee$ has rank $r$ with $\det (\Ee) \cong \Oo_X$ and $c_2(\Ee)=r$.
Let $\epsilon = (\epsilon _1,\dots ,\epsilon _m)\in \Ext_X^1(\Jj, \Oo _X^{\oplus m})$ be the extension class corresponding to (\ref{eqn2}) with $\epsilon _i\in  \Ext_X^1(\Jj, \Oo _X)$. Note that $h^0(\Ee)=m$ and $f(\Oo_X^{\oplus m})$ is the image of the evaluation map $\rho_{\Ee} : H^0(\Ee) \otimes \Oo_X\rightarrow \Ee$ with $\Jj=\mathrm{coker}(\rho_{\Ee})$.

By Lemma \ref{n2} and Serre's duality, we have $\mathrm{ext}_X^1(\Jj, \Oo_X )=h^1(\Jj)=m$. From the generality of $\epsilon $ we see that the extensions $\epsilon _1,\dots ,\epsilon _m$ are linearly independent. In particular, we have $A{\cdot}\epsilon \ne 0$ for all $A\in \mathrm{GL}(m)$, and so $\Ee \ncong \Oo _X\oplus \Gg$ with $\Gg$ an extension of $\Jj$ by $\Oo _X^{\oplus (m-1)}$. Since
$f(\Oo _X^{\oplus m}) \subset \Ee$ is the image of $\rho_{\Ee}$, we get that $\Ee \ncong  \Oo _X\oplus \Gg$ for any sheaf $\Gg$, i.e. $\Ee$ has no trivial factor.

Assume that $\Ee$ is decomposable, say $\Ee \cong \Ee _1\oplus \Ee _2$ with each $\Ee _i\ne 0$. Since the global section functor $H^0(-)$ and the evaluation map commute with direct sums, we have $\Jj \cong \mathrm{coker}(\rho_{\Ee _1})\oplus \mathrm{coker}(\rho_{\Ee _2})$. Since $\Jj$ is indecomposable, we get $\mathrm{coker}(\rho_{\Ee _i}) =0$ for some $i\in \{1,2\}$. This implies that $\Ee_i$ is trivial, which is impossible because $\Ee$ has no trivial factor.

To conclude the case $r$ even we need to find a sheaf $\Ee$ that is locally free and aCM. Consider the variety $T(S_1,\dots ,S_m)$ together with the sheaf $\Ss (S_1,\dots, S_m)$ in Lemma \ref{n5}. Define
$$\Vv(S_1, \dots, S_m):=\mathcal{E}xt_{p_2}^1(\Ss(S_1, \dots, S_m), p_2^*\Oo_X^{\oplus m})$$
to be the relative $\mathcal{E}xt_{p_2}^1$-sheaf as in \cite[Proposition 3.1]{Lange}; the fibre of $\Vv(S_1, \dots, S_m)$ over a point $\Jj\in T(S_1, \dots, S_m)$ is the set of all extensions of $\Jj$ by $\Oo_X^{\oplus m}$. By Lemma \ref{n2} the sheaf $\Vv(S_1, \dots, S_m)$ is a vector bundle of rank $m^2$ on $T(S_1, \dots, S_m)$ and so it is an affine space. Since $\Gg _{S_1}\oplus \cdots \oplus \Gg _{S_m}$ is locally free and aCM, the sheaf associated to a general point in $\Vv$ is locally free and aCM. Define
$$\UU:=\left\{ (S_1, \ldots, S_m) ~|~S_i \subset X_{\mathrm{reg}} \text{ with } |S_i|=2 \text{ and }S_i \cap S_j=\emptyset \text{ for all } i\ne j\right \}$$
and consider a vector bundle $\Vv$ on $\UU$, whose fibre over $(S_1, \dots, S_m)$ is $\Vv (S_1, \dots, S_m)$. Then there exists a non-empty open subset $\Vv' \subset \Vv$ such that the middle term of each extension in $\Vv'$ is aCM and locally free. As in the proof of Theorem \ref{k2} we can choose an irreducible subvariety $\Gamma \subset \Vv'$ such that the restriction of the map $\Vv '\rightarrow \UU$ to $\Gamma$ is quasi-finite and dominant. Hence we get the assertion for the case $r$ even.

Now assume that $r$ is odd, say $r=2m+3$. The case $m=0$ is true by Proposition \ref{k1} with $r=3$, because we have $g=h^0(\Oo _X(1)) \ge 3$. Now assume $r\ge 5$, i.e. $m\ge 1$, and that Theorem \ref{k3} is true for all odd integers less than $r$. We fix subsets $S_0, \dots, S_m\subset X_{\mathrm{reg}}$ with $|S_0|=3$ and $|S_i|=2$ for all $i\ge 1$ such that $S_i \cap S_j=\emptyset$ for all $i\ne j$. Define
$$\Ww(S_1, \dots, S_m;  S_0):=\mathcal{E}xt_{p_2}^1(\tilde{\Ss}(S_1, \dots, S_m; S_0), p_2^*\Oo_X^{\oplus (m+2)}),$$
where $\tilde{\Ss}(S_1, \dots, S_m; S_0)$ is the universal sheaf in Remark \ref{rte}. Then it parametrizes all the extensions of some sheaf $\tilde{\Jj}\in \JJ(S_1, \dots, S_m; S_0)$ by $\Oo_X^{\oplus (m+2)}$. Note that for each extension in $\Ww (S_1, \dots, S_m; S_0)$ the corresponding middle term $\Ee$ is torsion-free and has rank $r=2m+3$ with $\det (\Ee)\cong \Oo_X$ and $c_2(\Ee)=r$.

Let us denote by $\Gg_{S_0}$ an aCM and indecomposable vector bundle of rank three, admitting an extension of $\Ii_{S_0, X}$ by $\Oo_X^{\oplus 2}$ as in Proposition \ref{k1}. Then $\oplus_{i=1}^m \Gg_{S_i}$ is the middle term of an extension in $\Ww (S_1, \dots, S_m; S_0)$, which is locally free and aCM. So the general extension in $\Ww(S_1, \dots, S_m; S_0)$ has the aCM and indecomposable middle term. Now fix an indecomposable sheaf $\tilde{\Jj}\in \JJ (S_1, \dots, S_m; S_0)$ in Remark \ref{ees1} and consider a general sheaf $\Ee$ fitting into the following exact sequence:
\begin{equation}\label{uio}
0\to \Oo_X^{\oplus (m+2)} \stackrel{f}{\to} \Ee \stackrel{g}{\to} \tilde{\Jj} \to 0.
\end{equation}
Assume that $\Ee$ is decomposable, say $\Ee \cong \Ee_1\oplus \Ee_2$ with each $\Ee_i\not\cong 0$. As before, $f(\Oo_X^{\oplus (m+2)})$ is the image of the evaluation map $\rho_{\Ee}: H^0(\Ee)\otimes \Oo_X \rightarrow \Ee$ and $\mathrm{coker}(\rho_{\Ee})=\tilde{\Jj}$. Since the global section functor $H^0(-)$ and the evaluation map commute with finite direct sums, we have $\tilde{\Jj} \cong \mathrm{coker}(\rho_{\Ee_1})\oplus \mathrm{coker}(\rho_{\Ee_2})$. Since $\tilde{\Jj}$ is indecomposable, we get that $\Ee_i$ is trivial for some $i$, which contradicts to the generality of the extension (\ref{uio}), because we have $\mathrm{ext}_X^1(\tilde{\Jj}, \Oo_X)=h^1(\tilde{\Jj})=m+2$ by Remark \ref{en2}. As in the case $r$ even, we define
\begin{align*}
\tilde{\UU}:=\{ (S_0, S_1, \ldots, S_m) ~|~ &S_i \subset X_{\mathrm{reg}} \text{ with } |S_0|=3, \\
&|S_i|=2 \text{ for all $1\le i \le m$ and }S_i \cap S_j=\emptyset \text{ for all } i\ne j\}.
\end{align*}
We consider a vector bundle $\Ww$ on $\UU$, whose fibre over $(S_0, S_1, \dots, S_m)$ is $\Ww (S_1, \dots, S_m; S_0)$. Then we get the assertion, following the same argument in the case $r$ even.
\end{proof}

\begin{remark}\label{lle}
Let $\pi : Y \rightarrow X$ be a birational morphism between integral projective surfaces with $\omega_X \cong \Oo_X$ and $q(X)=0$ such that $\pi$ induces an isomorphism $\pi^{-1}(X_{\mathrm{sing}})\cong X_{\mathrm{sing}}$. In particular, we have $Y_{\mathrm{reg}}=\pi^{-1}(X_{\mathrm{reg}})$. This implies that $\pi _\ast \Oo_Y\cong \Oo _X$ and $R^1\pi _\ast \Oo _Y \cong 0$. Since each fiber of $\pi$ has dimension at most one, we also have $R^2\pi _\ast \Ff \cong 0$ for any coherent sheaf $\Ff$ on $X$. Thus we have $q(Y)=0$ and $h^2(\Oo _Y)=1$. Since $\pi$ induces an isomorphism between $\pi^{-1}(X_{\mathrm{sing}})$ and $X_{\mathrm{sing}}$, the canonical sheaf $\omega _Y$ is locally free with $h^0(\omega_Y)=1$ and so there is an effective divisor $\Delta$ such that $|\omega _Y| = \{\Delta \}$; we have $\Delta =\emptyset$ if and only if $\pi$ is an isomorphism. By Serre's duality we have $\mathrm{ext}_Y^1(\Ii _{S,Y}, \Oo_Y)= h^1(\Ii _{S,Y}\otimes \omega _Y)$. Since $|\omega _Y|  = \{\Delta \}$ and $S\cap \Delta =\emptyset$, we may use the long exact sequence of cohomology of the following exact sequence
\[
\pushQED{\qed}
0\to \Ii_{S,Y}\otimes \omega_Y \to \omega_Y \to \Oo_S \to 0
\]
to obtain $\mathrm{ext}_Y^1(\Ii _{S,Y}, \Oo_Y)=|S| -1$ for any finite subset $S\subset Y_{\mathrm{reg}}\setminus \Delta$. Then the same statement of Theorem \ref{k3} holds for $Y$, using the same argument in its proof with subsets $S_i\subset Y_{\mathrm{reg}}\setminus \Delta$ for $i=0,\cdots, m$.
\end{remark}

\section{Enriques surfaces}
In this section we assume that $X$ is an integral projective surface with $q(X)=0$ and $\omega _X\ncong \Oo_X$ locally free such that $h^0(\omega_X)=0$ and $h^0(\omega_X^{\otimes 2})=1$. Let $\Delta \ge 0$ be the effective divisor such that $\omega _X^{\otimes 2}\cong \Oo _X(\Delta )$. When $X$ is smooth, the minimal model of $X$ is an Enriques surface. Note that $h^2(\Oo _X)=h^0(\omega _X) =0$ and so $\chi(\Oo_X)=1$. Set $X':= X_{\mathrm{reg}}\cap (X\setminus \Delta)$.

\begin{remark}
We fix an ample line bundle $\Oo _X(1)$ on $X$ such that $h^1(\Oo _X(t)) =0$ for all $t\in \ZZ$; at least in characteristic zero Kodaira's vanishing theorem shows that we only need this assumption for $t\ge 0$. The case $t=0$ is a general assumption of the surfaces considered in this article. Serre's duality gives $h^1(\omega _X(t)) =0$ for all $t\in \ZZ$. Notice that using Riemann-Roch it is easy to see that under these hypothesis $h^0(\omega _X(1)) \ne 0$. In summary, we take a polarization $\Oo _X(1)$ such that $h^0(\omega _X(1)) \ne 0$ and $h^1(\Oo _X(t)) =h^1(\omega _X(t)) =0$ for all $t\in \ZZ$. If $\Delta =\emptyset$, e.g. minimal Enriques surfaces, then we always have $h^1(\Oo_X(t)) =0$ for $t>0$, because $\omega _X(t)$ with $t>0$ is ample; it is numerically equivalent to $\Oo_X(t)$ and so we can use Kodaira's vanishing theorem.
\end{remark}

For any point $p\in X_{\mathrm{reg}}$, we have $\mathrm{ext}_X^1(\Ii_{p,X}, \Oo_X) = h^1(\Ii _{p,X}\otimes \omega _X)=1$ by Serre's duality. Thus, up to isomorphisms, there is a unique sheaf $\Ee _p$ that fits into the following non-trivial extension:
\begin{equation}\label{eqee1}
0\to \Oo_X \to \Ee_p \to \Ii_{p,X} \to 0.
\end{equation}
Obviously $\Ee _p$ has rank two and it is locally free outside $p$ with $\det (\Ee_p) \cong \Oo_X$. Since $p\in X_{\mathrm{reg}}$ and $h^0(\omega _X)=0$, the Cayley-Bacharach condition is satisfied. Thus $\Ee _p$ is locally free. Note that the point $p$ is uniquely determined by the isomorphism class of $\Ee_p$, because we have $h^0(\Ee_p)=1$ by the sequence (\ref{eqee1}) and any non-zero section of $\Ee_p$ vanishes only at $p$.

\begin{lemma}\label{e1}
For a general $p\in X_{\mathrm{reg}}$ the vector bundle $\Ee _p$ is aCM and indecomposable.
\end{lemma}

\begin{proof}
The exact sequence (\ref{eqee1}) twisted by $\Oo_X(t)$ gives $h^1(\Ee _p(t)) =0$ for all $t\ge 0$. From $\Ee _p^\vee \cong \Ee _p$ we see that $h^1(\Ee _p\otimes \omega _X) = h^1(\Ee _p)=0$ by Serre's duality. Now fix an integer $t<0$. The twist of the sequence (\ref{eqee1}) by $\omega_X(-t)$ gives
$$h^1(\Ee _p\otimes  \omega _X(-t)) \le h^1(\omega _X(-t)) +h^1(\Ii_{p,X}\otimes \omega _X(-t))=h^1(\Ii_{p,X}\otimes \omega _X(-t)).$$
Here, we have $h^1(\omega _X(-t))=0$ by our assumptions on the polarization $\Oo_X$. We also have $h^0(\omega_X(-t))>0$ from the assumption that $h^0(\omega_X(1))>0$. Since $p$ is general, we have $h^1(\Ii_{p,X}\otimes \omega _X(-t))=0$. By Serre's duality, this implies that $h^1(\Ee_p(t))=h^1(\Ee_p\otimes \omega_X(-t))=0$. Thus $\Ee _p$ ia aCM.

Assume that $\Ee _p$ is decomposable; say $\Ee _p\cong \Aa_1 \oplus \Aa_2$ with each $\Aa_i$ a line bundle. Since $h^0(\Ee _p) =1$, we may assume that $h^0(\Aa_1 )=1$ and $h^0(\Aa_2)=0$. Since the evaluation map commutes with direct sums and $\Ii _{p,X}$ is isomorphic to the cokernel of the evaluation map $H^0(\Ee _p)\otimes \Oo _X\rightarrow \Ee _p$, we get $\Aa_2 \cong \Ii _{p,X}$, a contradiction.
\end{proof}

\begin{lemma}\label{e1.1}
For any two general points $p, q\in X_{\mathrm{reg}}$, we have $\mathrm{ext}_X^1(\Ee _p,\Ee _q) =1$.
\end{lemma}

\begin{proof}
Since $\det (\Ee _q) \cong \Oo _X$, we have $\Ee _q^\vee \cong \Ee _q$ and so $\Ext _X^1(\Ee _p,\Ee _q) \cong H^1(\Ee _p\otimes \Ee _q)$. Tensoring the exact sequence (\ref{eqee1}) with $\Ee _q$, we get the exact sequence
\begin{equation}\label{eqe1.0}
0 \to \Ee _q \to \Ee _p\otimes \Ee _q\to \Ii _{p,X}\otimes \Ee _q\to 0.
\end{equation}
Since $\Ee _q$ is aCM, we have $h^1(\Ee _q)=0$. On the other hand, tensoring the sequence (\ref{eqee1}) for $\Ee_q$ with $\omega _X$ gives $h^0(\Ee _q\otimes \omega _X)=0$, because $\omega _X\ncong \Oo _X$. Thus by Serre's duality we get $h^2(\Ee _q) = h^0(\Ee _q\otimes \omega _X)=0$. Then the assertion follows from the exact sequence
$$0\to \Ii _{p,X} \otimes \Ee _q\to \Ee_q\to (\Ee _q)_{|\{p\}}\to 0$$
together with the fact that $\Ee_q$ is an aCM vector bundle of rank two and $H^0(\Ee _q)$ is one-dimensional whose nontrivial section vanishes only at $q$ so that $h^0(\Ii_{p,X}\otimes \Ee_q)=0$.
\end{proof}

\begin{proposition}\label{e2}
Setting $\tilde{g}:=h^0(\Oo_X(1))$, there exists an indecomposable aCM vector bundle $\Ee$ of rank $r$ on $X$ with $\det (\Ee )\cong \Oo _X$ and $c_2(\Ee )=r-1$ for each integer $2\le r \le \tilde{g}-1$.
\end{proposition}

\begin{proof}
As in the proof of Proposition \ref{k1}, consider a general sheaf $\Ee$ fitting into the sequence (\ref{eqk1}) for a general $S\subset X_{\mathrm{reg}}$ with $|S|= r-1$. Then we get $\mathrm{ext}_X^1(\Ii_{S,X}, \Oo_X)=r-1$ and the proof of Proposition \ref{k1} works verbatim.
\end{proof}

\begin{theorem}\label{enriquesmain}
Let $X$ be an integral projective surface with $q(X)=0$ and $\omega _X\ncong \Oo_X$ locally free such that $h^0(\omega_X)=0$ and $h^0(\omega_X^{\otimes 2})=1$. Then for any $r\geq 2$ there exists a family $\{\Ee _\alpha\}_{\alpha \in \Gamma}$ of dimension $2\ceil{\frac{r}{2}}$ of indecomposable rank $r$ aCM vector bundles  with $c_1(\Ee_\alpha)\cong\Oo_X$ such that for each $\alpha \in \Gamma$ there are only finitely many $\beta \in \Gamma$ with $\Ee _\beta \cong \Ee_\alpha$.
\end{theorem}
\begin{proof}
The proof follows exactly the same structure as in the case of Theorem \ref{k3}. In the present setting, however, in the case of even rank $r=2m$, the family $\Gamma$ of indecomposable aCM vector bundles of rank $r$ will be mapped by a quasi-finite dominant morphism to
$$\UU:=\left\{ (S_1, \ldots, S_m) ~|~S_i \subset X_{\mathrm{reg}} \text{ with } |S_i|=1 \text{ and }S_i \cap S_j=\emptyset \text{ for all } i\ne j\right \},$$
\noindent a variety of dimension $r$, while in the odd case $r=2m+3$ it will be mapped to
\begin{align*}
\widetilde{\UU}:=\{ (S_0, S_1, \ldots, S_m) ~|~ &S_i \subset X_{\mathrm{reg}} \text{ with } |S_0|=2, \\
&|S_i|=1 \text{ for all $1\le i \le m$ and }S_i \cap S_j=\emptyset \text{ for all } i\ne j\}.
\end{align*}
\noindent a variety of dimension $2m+4=2\ceil{\frac{r}{2}}$.
\end{proof}


\section{Irregular surfaces}

In this section we deal with surfaces with $q(X)\ge 1$.

\begin{proposition}\label{c1}
Let $X$ be a smooth projective surface with $q(X) =1$ and a fixed ample line bundle $\Oo _X(1)$, satisfying one of the following conditions:
\begin{itemize}
\item [(i)] $\Oo _X(1) \cong \omega _X$;
\item [(ii)] $\Oo _X(1)\otimes \omega _X^\vee$ is ample.
\end{itemize}
Then for each positive integer $r$ there exists a one-dimensional family $\{\Ee _\alpha \}_{\alpha \in \Gamma}$ of indecomposable aCM vector bundles of rank $r$ on $X$ such that $\Ee_{\alpha}$ for each $\alpha \in \Gamma$ is strictly semistable with $\det (\Ee _\alpha )\in \mathrm{Pic}^0(X)$ and $c_2(\Ee _\alpha )=0$ with respect to any polarization of $X$, and there are only finitely many $\beta \in \Gamma$ with $\Ee _\beta \cong \Ee _\alpha$.
\end{proposition}

\begin{proof}
Fix a general line bundle $\Ll\in \mathrm{Pic}^0(X)$. Then we have $h^1(\Ll)=0$; see \cite[Th. 0.1]{beau}, \cite[Theorem 1]{gl1} or \cite[Theorem 0.1]{gl2}. We also have $h^1(\Ll(-t))=0$ for all $t>0$ by Kodaira's vanishing. Note that Serre's duality gives $h^1(\Ll(t)) = h^1(\Ll^\vee \otimes \omega _X(-t))$. Then we have $h^1(\Ll^\vee \otimes \omega_X(-t))=0$ for all $t>0$. Indeed, in case (i) we may apply Kodaira's vanishing for $t\ge 2$ and $h^1(\Ll^\vee)=0$ for $t=1$. In case (ii) $\omega_X^\vee (t)$ is ample and so we may apply Kodaira's vanishing. Thus $\Ll$ is aCM.

Let $\phi : X \rightarrow C$ be the Albanese map of $X$ onto an elliptic curve $C$. We have $\phi_\ast \Oo _X \cong \Oo _C$ and $\mathrm{Pic}^0(X) = \phi^\ast \mathrm{Pic}({C})$. By the classification of vector bundles on an elliptic curve in \cite{atiyah:elliptic}, there is an indecomposable vector bundle $\Ff$ of rank $r$ on $C$, which is an iterated extension of $\Oo_C$. Define
$$\Ee _{\Ll}:= \phi^\ast \Ff \otimes \Ll.$$
Then $\Ee_{\Ll}$ is a vector bundle of rank $r$ on $X$ with $\det (\Ee _{\Ll})\cong \Ll^{\otimes r}\in \mathrm{Pic}^0(X)$ and $c_2(\Ee _{\Ll}) =0$, which is an iterated extension of $\Ll$. Since $\Ll$ is aCM, so is $\Ee _{\Ll}$.

Assume that $\Ee _{\Ll}$ is decomposable and this would imply that $\phi^\ast \Ff$ is also decomposable, say $\phi ^\ast  \Ff \cong \Ff_1 \oplus \Ff_2$ with each $\Ff_i$ an aCM vector bundle of rank $r_i$ with $0<r_i<r$. By the projection formula and $\phi_\ast \Oo _X \cong \Oo _C$, we have $\Ff \cong \phi_\ast \Ff_1 \oplus \phi_\ast \Ff_2$. Now take a non-empty subset of $C$ so that
\begin{itemize}
\item we have $\Ff_{|U} \cong \Oo_U^{\oplus r}$, and
\item $\phi^{-1}(q)$ is a smooth projective curve for each $q\in U$.
\end{itemize}
Since $(\phi^\ast \Ff)_{|\phi^{-1}(q)}$ is the trivial vector bundle of rank $r$ on the integral projective curve $\phi^{-1}(q)$, we get $\Ff_{i|\phi^{-1}(q)} \cong \Oo _{|\phi^{-1}(q)}^{\oplus r_i}$ for each $i$. In particular, we have $\phi_*\Ff_i$ is not zero for each $i$, a contradiction to the indecomposability of $\Ff$.
\end{proof}

\begin{remark}\label{oop1}
Let $X$ be a smooth and connected projective variety of dimension $n\ge 2$ and $\phi : X\rightarrow \mathrm{Alb}(X)$ its Albanese map. Assume that $X$ has {\it maximal Albanese dimension}, i.e. $\dim \phi (X) =n$. Note that this implies $q(X) = \dim \mathrm{Alb}(X) =n\ge 2$. In particular, an abelian variety has maximal Albanese dimension. Let $\Oo _X(1)$ be an ample line bundle on $X$ such that $\omega _X ^\vee \otimes \Oo _X(1)$ is ample; if $X$ is an abelian variety, then $\Oo _X(1)$ can be arbitrary.

Now choose a general line bundle $\Ll\in \mathrm{Pic}^0(X)$. Since $X$ has Albanese dimension $n$, we have $h^i(\Ll) =0$ for all $1\le i \le n-1$ by \cite[Theorem 1]{gl1} or \cite[Theorem 0.1]{gl2}. Fix a positive integer $t$. By Kleiman's numerical criterion of ampleness in \cite{k}, we get that $\Ll^\vee(t)$ and $\omega _X^\vee \otimes \Ll(t)$ are ample for $t>0$. Then Kodaira's vanishing gives $h^i(\Ll(-t)) = h^i(\omega_X \otimes \Ll^\vee(-t) ) =0$ for all $1\le i \le n-1$. On the other hand,  Serre's duality gives $h^i(\Ll(t)) = h^{n-i}(\omega _X\otimes \Ll^\vee (-t)) =0$ for $1\le i \le n-1$. This implies that $\Ll$ is aCM. Since $\dim \mathrm{Pic}^0(X)=q(X)$, there exists a $n$-dimensional family of pairwise non-isomorphic aCM lines bundles.
\end{remark}

Now we work on the proof of Theorem \ref{cc1} and the key tool is Mukai's study of vector bundles on abelian varieties; see \cite{muk}.

\begin{proof}[Proof of Theorem \ref{cc1}:]
Since $X$ is smooth and birational to an abelian variety, there are an $n$-dimensional abelian variety $Y$ and a proper birational morphism $v: X\rightarrow  Y$; see \cite[Proposition 9.12]{ueno}. In particular, we have $v_\ast \Oo _X \cong \Oo _Y$ by the Zariski Main Theorem in \cite[Corollary III.11.4]{Hartshorne}). Let $\widehat{Y} =\mathrm{Pic}^0(Y)$ denote the abelian variety dual to $Y$. As in \cite[Definitions 4.4, 4.5, 4.6]{muk} we consider the following set
\[
\UU '_r:=\{ \text{ the unipotent vector bundles of rank $r$ on $Y$}\},
\]
i.e. the set of all vector bundles of rank $r$ on $Y$, obtained by $(r-1)$-times of iteration of $\Oo _Y$; we have $\UU '_1 = \{\Oo _Y\}$ and $\UU '_r$ is the set of all vector bundles which admit extensions of $\Oo _Y$ by an element of $\UU '_{r-1}$. If we let $R$ be the completion of the local ring $\Oo _{\widehat{Y},0}$ and $B_f$ the set of all $R$-modules with finite length, then by \cite[Theorem 4.12]{muk} there is a bijection between $\UU '_r$ and the set $B_f[r]$ of the $R$-modules of length $r$. Note that this bijection preserves finite direct sums. Thus to an indecomposable vector bundle in $\UU'_r$ it is enough to consider an indecomposable elements of $B_f[r]$. Define a subset
\[
\UU_r:=\Bigg \{ \Aa \in \UU'_r~\Bigg|~ \makecell{\Aa \text{ corresponds to an indecomposable elements of $B_f[r]$}\\ \text{ of the form $R/I$ with $I\subset R$ an ideal of colength $r$}}\Bigg\},
\]
consisting of elements of the local Hilbert scheme of $R$ corresponding to connected zero-dimensional subschemes of $\widehat{Y}$ of degree $r$ with $0$ as their support. Then we get an algebraic family $\UU _r$ of indecomposable unipotent vector bundles of rank $r$. For the known results on the dimension of $\UU_r$, refer to \cite[page 6]{g}. For $n=2$ and arbitrary $r$, $\UU_r$ is irreducible of dimension $r-1$ by \cite{b, i}, while it can be reducible for $n\ge 3$ by \cite{g, i}. In any case with $n\ge 2$, $\UU_r$ has an irreducible family of dimension $(n-1)(r-1)$, whose general element is curvilinear, or collinear, by \cite[pages 5--6]{g}.

For any line bundle $\Ll \in \mathrm{Pic}^0(X)$, set
$$\Theta _{\Ll}:=\{v^\ast (\Ff) \otimes \Ll ~|~\Ff \in \UU _r\}.$$
Each element of $\Theta _{\Ll}$ is a vector bundle of rank $r$ on $X$, which is an iterated extension of $\Ll$. Thus each element of $\Theta _{\Ll}$ is strictly semistable with respect to any polarization on $X$ and all its Chern classes are zero. Assume that $v^\ast (\Ff )\otimes \Ll \cong v^\ast (\Gg )\otimes \Ll$ for $\Ff, \Gg \in \UU _r$. Then we get $v^\ast (\Ff )\cong v^\ast (\Gg)$ and so $\Ff \cong \Gg$ by the projection formula and $v_\ast \Oo _X \cong \Oo _Y$. In particular, $\Theta _{\Ll}$ parametrizes one-to-one vector bundles of rank $r$ on $X$ and $\dim \Theta _{\Ll} = \dim \UU _r$. Note that for each $\Aa \in \Theta_{\Ll}$ there are only finitely many $\Ll' \in \mathrm{Pic}^0(X)$ such that $\Aa \cong \Aa'$ for some $\Aa' \in \Theta_{\Ll'}$; indeed, we have at most $(2n)^r$ vector bundles $\Aa'$, because $\det (\Aa )\cong \Ll^{\otimes r}$ and so $\Ll'\otimes \Ll^\vee$ is an element of $r$-torsion of $\mathrm{Pic}^0(X)$. Now a general line bundle $\Ll \in \mathrm{Pic}^0(X)$ is aCM by Remark \ref{oop1}. Define a non-empty open subset
$$\VV:=\{ \Ll \in \mathrm{Pic}^0(X)~|~\Ll \text{ is aCM }\},$$
which is an algebraic variety of dimension $q(X)=n$. For each $\Ll \in \VV$, every vector bundle $\Aa\in \Theta _{\Ll}$ is aCM, because it is an iterated extension of aCM vector bundles. Define a parameter space $\Gamma$ over $\VV$ whose fibre over $\Ll$ is $\Theta_{\Ll}$. Then it is a parameter space, finite-to-one, for indecomposable aCM vector bundles of rank $r$ on $X$ with $\dim \Gamma = n+\dim \UU _r =(n-1)r+1$.
\end{proof}

\begin{proposition}\label{c2}
Let $X$ be a smooth projective surface with $q(X) \ge 2$ and a fixed ample line bundle $\Oo _X(1)$ satisfying one of the following conditions:
\begin{itemize}
\item [(i)] $\Oo _X(1) \cong \omega _X$;
\item [(ii)] $\Oo _X(1)\otimes \omega _X^\vee$ is ample.
\end{itemize}
Then for each integer $r$ with $1\le r \le q(X)$ there exists a $q(X)$-dimensional family $\{\Ee _\alpha \}_{\alpha \in \Gamma}$ of indecomposable aCM vector bundles of rank $r$ on $X$ such that $\Ee_{\alpha}$ for each $\alpha \in \Gamma$ is strictly semistable with $\det (\Ee _\alpha )\in \mathrm{Pic}^0(X)$ and $c_2(\Ee _\alpha )=0$ with respect to any polarization of $X$, and there are only finitely many $\beta \in \Gamma$ with $\Ee _\beta \cong \Ee _\alpha$.
\end{proposition}

\begin{proof}
Fix a general line bundle $\Ll\in \mathrm{Pic}^0(X)$. Then as in Remark \ref{oop1} we see that $\Ll$ is aCM. Set $\Gg_0=0$ the zero sheaf and $\Gg _1:= \Ll$. For an integer $r\ge 2$, we define $\Gg_{r}$ inductively as a general sheaf fitting into the following extension
\begin{equation}\label{eer2}
0\to \Gg_{r-1} \stackrel{u}{\to} \Gg_r \stackrel{v}{\to} \Ll \to 0.
\end{equation}
Note that $\Gg_r$ is strictly semistable for any polarization and $\Gg_r\otimes \Ll^\vee$ is an iterated extension of $\Oo_X$ for each $r\ge 1$. Since $\Gg _{r-1}\otimes \Ll^\vee$ is an iterated extension of $\Oo _X$, we have $\det (\Gg _{r-1}\otimes \Ll^\vee)\cong \Oo _X$ and $c_2(\Gg _{r-1}\otimes \Ll^\vee) =0$. Moreover, we may choose $\Gg_r$ admitting a non-trivial extension (\ref{eer2}), because we have $\mathrm{ext}_X^1(\Ll, \Gg_{r-1})>0$; indeed, we have $h^1(\Gg _{r-1}\otimes \Ll^\vee ) \ge q(X)-r+2$, which is clearly true for $r=2$. In general, we get the following exact sequence from (\ref{eer2})
$$H^0(\Oo _X)\to H^1(\Gg _{r-1}\otimes \Ll^\vee )\to H^1(\Gg _r\otimes \Ll^\vee ).$$
Then we may apply the inductive hypothesis and $h^0(\Oo_X)=1$.

Note that the coboundary map $H^0(\Oo _X)\rightarrow H^1(\Gg _{r-1}\otimes \Ll^\vee )$ is zero if and only if (\ref{eer2}) is the trivial extension. Since we take a non-trivial extension at each step, we have $h^0(\Gg _r\otimes \Ll^\vee ) =h^0(\Gg _{r-1}\otimes \Ll^\vee)$. By induction on $r$ we get $h^0(\Gg _r\otimes \Ll^\vee )=1$ for all $r\le q(X)$. Assume now that $\Gg_r$ is decomposable, say $\Gg_r \cong \Ff_1\oplus \Ff_2$ with each $\Ff_i$ nonzero. Then each $\Ff_i \otimes\Ll^\vee $ is a strictly semistable vector bundle with numerically trivial determinant. Since $gr (\Gg _{r-1}\otimes \Ll^\vee ) =\Oo _X^{\oplus (r-1)}$, we that $gr (\Ff _i\otimes \Ll^\vee)$ is trivial and so each $\Ff _i\otimes \Ll^\vee$ has a subsheaf isomorphic to $\Oo _X$. In particular, we have $h^0(\Gg _r\otimes \Ll^\vee )\ge 2$, a contradiction.

Note that $\det (\Gg _r) \cong \Ll^{\otimes r}$ and so there are only finitely many line bundles $\Ll' \in \mathrm{Pic}^0(X)$ such that $\Gg _r$ is also an iterated extension of $\Ll'$. Hence we get the assertion from $\dim \mathrm{Pic}^0(X)=q(X)$.
\end{proof}

\begin{remark}
Let $Y$ be a hyperelliptic surface, i.e. a smooth projective surface with $\omega _Y\ncong \Oo _Y$, $q(Y)=1$ and $\omega _Y^{\otimes 12}\cong \Oo _Y$. In particular, we have $h^2(\Oo_Y) =h^0(\omega _Y)=0$ and so $\chi (\Oo _Y) =0$. Let $X$ be a smooth projective surface birational to $Y$. Then we have $h^i(\Oo_X)=h^i(\Oo_Y)$ for each $i$ and $\omega _X\ncong \Oo _X$ with $h^0(\omega_X^{\otimes 12})=1$. Fix an ample line bundle $\Oo _X(1)$ on $X$ and take a line bundle $\Ll \in \mathrm{Pic}^0(X)\setminus \{\Oo_X, \omega_X^\vee\}$. Then we have $h^0(\Ll)=h^2(\Ll) =0$. Since $\Ll$ is numerically equivalent to $\Oo_X$ and $\chi (\Oo _X) =0$, we have $\chi (\Ll)=0$ and so $h^1(\Ll)=0$. Note that $\Ll(t)$ and $\Ll^\vee \otimes \omega _X(t) $ are ample for $t>0$, because they are numerically equivalent to the ample line bundle $\Oo _X(t)$. So we get $h^1(\Ll(t)) =0$ for all $t\ne 0$ by Kodaira's vanishing and Serre's duality. Thus $\Ll$ is aCM. Now we may construct indecomposable aCM vector bundles $\Gg_r$ of rank $r$ as in the case of abelian surfaces. Indeed, we have $\mathrm{ext}_X^1(\Ll, \Ll)=h^1(\Oo_X)=1$ and $\mathrm{ext}_X^1(\Ll, \Gg_{r-1})>0$. We have $\det (\Gg _r) \cong \Ll^{\otimes r}$. In particular, there are only finitely many line bundles $\Ll'\in \mathrm{Pic}^0(X)$ such that $\Gg_r$ is an iterated extension of $\Ll'$. We get the following result from $q(X)=1$.
\end{remark}
\begin{proposition}\label{hyperelliptic}
Let $X$ be a smooth projective surface, birational to a hyperelliptic surface, with any polarization. For any positive integer $r$, there exists a one-dimensional family $\{\Ee _\alpha \}_{\alpha \in \Gamma}$ of indecomposable aCM vector bundles of rank $r$ on $X$ such that for each $\alpha \in \Gamma$ there are only finitely many $\beta \in \Gamma$ with $\Ee _\beta \cong \Ee _\alpha$.
\end{proposition}


\section{Surfaces of general type with ample canonical line bundle}
Let $X$ be an integral projective surface, possibly singular, with ample $\omega_X$ satisfying the following conditions:
\begin{itemize}
\item [(i)] $h^1(\omega _X^{\otimes n}) =0$ for all $n\in \ZZ$;
\item [(ii-$\epsilon$)] $p_g:=h^0(\omega _X)\ge 2+\epsilon$ with $\epsilon \in \{0,1\}$.
\end{itemize}
We set $\Oo _X(1):= \omega _X$ with respect to which we consider aCM vector bundles on $X$.

\begin{remark}
Assume that $X$ is smooth. The canonical line bundle $\omega_X$ is ample if and only if $X$ is a minimal surface of general type without $(-2)$-curves, i.e. a smooth surface of general type without smooth rational curves $D\subset X$ with either $D^2=-1$ or $D^2=-2$; see \cite{bhpv}. There are surfaces $X$ of general type with $p_g= h^0(\omega _X)\le 1$, but most surfaces have $p_g\ge 2$. The condition (i) for $n=0$ is $h^1(\Oo _X)=0$, i.e. the irregularity of $X$ is $q(X)=0$. This is a non-trivial requirement, but it is satisfied in many important cases. By Serre's duality this would imply that $h^1(\omega _X)=q(X)=0$. In characteristic $0$ the condition (i) for $n<0$ comes from Kodaira's vanishing theorem by the ampleness of $\omega _X$. Assume $h^1(\omega _X^{\otimes n}) = 0$ for all $n<0$. By Serre's duality we have $h^1(\omega _X^{\otimes n}) =h^1(\omega _X^{\otimes (1-n)}) =0$ for $n\ge 2$. Thus in characteristic $0$ we have the condition (i) satisfied if and only if $h^1(\Oo _X)=0$.
\end{remark}

By the condition (ii-$\epsilon$), the set
$$\Sigma :=\mathrm{Sing}(X) \cap \big(\text{the base locus of }|\omega_X|\big)$$
is a proper closed subset of $X$. By the same argument in Remark \ref{lle} using Serre's duality we get the following lemma.

\begin{lemma}\label{g1}
For a finite subset $S\subset X\setminus \Sigma$, we have $\mathrm{ext}_X^1(\Ii_{S,X}, \omega _X) =|S|-1$ and a general extension of $\Ii_{S,X}$ by $\omega_X$ is locally free.
\end{lemma}
\begin{proof}
For the first assertion, we may apply the same argument in Remark \ref{lle} using Serre's duality. The second assertion is clear, because the Cayley-Bacharach condition for $S$ and the linear system $|\Oo_X|$ is satisfied.
\end{proof}

\begin{proposition}\label{g3}
For a fixed integer $2\le r \le p_g$ and a general subset $S\subset X\setminus \Sigma$ with $|S| = r$, the general sheaf $\Ee$ fitting into an exact sequence
\begin{equation}\label{eqg1}
0 \to \omega _X^{\oplus (r-1)} \to \Ee \to \Ii _{S,X}\to 0
\end{equation}
is an indecomposable and aCM vector bundle of rank $r$.
\end{proposition}

\begin{proof}
We get that $\Ee$ is locally free with rank $r$ from Lemma \ref{g1}. Indeed, if $r\ge 3$, then $\Ee \cong \Gg \oplus \omega_X^{\oplus (r-2)}$ with a general extension $\Gg$ of $\Ii_{S,X}$ by $\omega_X$ is locally free. Then we may use openness of being locally free. Now since we have $\mathrm{ext}_X^1(\Ii _{S,X}, \omega_X) =r-1$ by Lemma \ref{g1}, the extension (\ref{eqg1}) is induced by a choice of a basis $\{e_1,\dots ,e_{r-1}\}$ of $\Ext_X ^1(\Ii_{S,X}, \omega _X)$. Thus the map $\phi : H^1(\Ii_{S,X}) \rightarrow H^2(\omega_X^{\oplus (r-1)}) \cong \mathbf{k}^{\oplus (r-1)}$ is bijective, and in particular we have $h^1(\Ee)=0$. Recall that we assume $\omega_X \cong \Oo_X(1)$. Then by the condition (i) we get $h^1(\omega_X(n))=0$ for all $n\in \ZZ$ and we get
$$0\to H^1(\Ee(n)) \to H^1(\Ii_{S,X}(n)) \to H^2(\omega_X(n))^{\oplus (r-1)}. $$
Assume first that $n$ is positive and this implies $h^2(\omega_X(n))=h^0(\Oo_X(-n))=0$. Since $S$ is general with $|S|=r\le h^0(\Oo_X(1))\le h^0(\Oo_X(n))$, we get $h^1(\Ii_{S,X}(n))=0$. Thus we have $h^1(\Ee(n))=0$. It remains to show that $h^1(\Ee(-n))=0$ for $n\ge 1$. In fact, it is sufficient to prove the existence of an extension $\Ff$ of $\Ii _{S,X}$ by $\omega _X^{\oplus (r-1)}$ satisfying $h^1(\Ff(-n)) =0$ for all $n\ge 1$. Take $\Ff \cong \Gg \oplus \omega _X^{\oplus (r-2)}$ with a general extension $\Gg$ of $\Ii _{S,X}$ by $\omega _X$ given by $e_1$. By the previous argument, we have $h^1(\Gg(n))=0$ for all $n\ge 1$. By Lemma \ref{g1}, $\Gg$ is locally free with $\det (\Gg) \cong \omega_X$. Serre's duality gives $h^1(\Gg(-n))=h^1(\Gg(n))=0$ for all $n\ge 1$. Thus we get that $\Ee$ is aCM. Note that if $r\ge 3$, then $\Gg$ is not aCM since we have $h^1(\Gg)=r-2$.

For the indecomposability, we may use the same argument in the proof of Proposition \ref{k1} to $\Ee \otimes \omega_X^{\vee}$, because $\Ii_{S,X}\otimes \omega_X^\vee$ is indecomposable.
\end{proof}

Now for the statement in Theorem \ref{thm22}, set $\epsilon =r-2\left \lfloor \frac{r}{2} \right\rfloor$ for which the condition (ii-$\epsilon$) for $X$ is assumed to be satisfied.
\begin{theorem}\label{thm22}
For each integer $r\ge 2$, there exists an $r$-dimensional family $\{\Ee _\alpha\}_{\alpha \in \Gamma}$ of indecomposable aCM vector bundles of rank $r$ on $X$ with $\det (\Ee _\alpha )\cong \omega_X^{\otimes \lceil r/2\rceil}$ and $c_2(\Ee _\alpha )=r$ such that for each $\alpha \in \Gamma$ there are only finitely many $\beta \in \Gamma$ with $\Ee _\beta \cong \Ee_\alpha$.
\end{theorem}
\begin{proof}
We use the same notations in the proof of Theorem \ref{k3} such as $\II(S_1, \dots, S_i)$ and $\JJ(S_1, \dots, S_i; S_0)$. Then we get the same assertions from Lemma \ref{n2} till Remark \ref{rte}; the only difference occurs in Lemma \ref{n4} and Remark \ref{ess}, where we have
$$\mathrm{ext}_X^1(\Ii_{S_{i+1}, X}, \Jj)=\mathrm{ext}_X^1(\Ii_{S_0, X}, \Jj)=i$$
for $\Jj \in \II(S_1, \dots, S_i)$ from $\mathrm{ext}_X^1(\Ii_{S_{i+1}, X}, \Oo_X)=\mathrm{ext}_X^1(\Ii_{S_0, X}, \Oo_X)=0$. Then we may consider the exact sequences (\ref{eqn2}) and (\ref{uio}) with $\Oo_X$ replaced by $\omega_X$.
\end{proof}


\section{Surfaces mapped to a curve of genus $\ge 3$ not as their Albanese image}\label{gege}

Throughout this section, $X$ is a smooth projective surface admitting a surjective map $v: X \rightarrow C$ with $g=g(C)\ge 3$. Assume that $C$ is such a curve achieving maximum possible genus $g$ and that $q(X)>g$. For example, we may take as $X$ any smooth surface birational to $C\times D$, where $D$ is a smooth curve with $1\le g(D) \le g$; in this case we have $q(X)= g+g(D)$.

\begin{proposition}\label{alb1}
For each positive integer $r$ there exists a family $\{\Ee _\alpha \}_{\alpha \in \Gamma}$ of indecomposable aCM vector bundles of rank $r$ on $X$ such that $\Gamma$ is an integral variety with
$$\dim \Gamma \ge q(X)+\frac{(r-1)(r-2)(g-1)}{2} -\frac{r(r-1)}{2}$$
and each $\Ee_{\alpha}$ is strictly semistable with $\det (\Ee _\alpha )\in \mathrm{Pic}^0(X)$ and
$c_2(\Ee _\alpha )=0$ with respect to any polarization of $X$ such that there are only finitely many $\beta \in \Gamma$ with $\Ee_\beta \cong \Ee _\alpha$.
\end{proposition}

Set $\Aa_1:=\Oo_C$ and define inductively a vector bundle $\Aa_{i+1}$ of rank $i+1$ on $C$ to be the middle term of the following extension:
\begin{equation}\label{eqalb0}
0 \to \Aa _i\to \Aa _{i+1}\to \Oo _C \to 0,
\end{equation}
where $\Aa_{i+1}=\Aa_{i+1}(e)$ corresponds to the extension class $e\in \Ext_C^1(\Oo_C, \Aa_i)\cong H^1(\Aa_i)$. The long exact sequence of cohomology of (\ref{eqalb0}) gives $g-1\le h^1(\Aa _{i+1})-h^1(\Aa _i) \le g$. Since we have $g\ge 3$ from the assumption, we get $h^1(\Aa_{i+1})\ne 0$. In particular, we may assume that the extension (\ref{eqalb0}) is non-trivial. The image of the coboundary map $H^0(\Oo _C)\rightarrow H^1(\Aa _i)$ corresponds to the extension (\ref{eqalb0}), up to a sign, and so we get $h^0(\Aa _{i+1}) = h^0(\Aa _i)$ and $h^1(\Aa _{i+1}) =h^1(\Aa _i)+g-1$ for each $i$. By induction, we get
\[
h^0(\Aa_i)=1~\text{ and }~h^1(\Aa_i)=i(g-1)+1.
\]
Note that each $\Aa _i$ is an iterated extension of $\Oo _C$, and in particular it is strictly semistable with $gr (\Aa _i)\cong \Oo _C^{\oplus i}$. Assume $\Aa _i\cong \Bb_1\oplus \Bb_2$ with each $\Bb_i\ne 0$. Since each $\Bb_i$ has a HN-filtration with $\Oo _C$ as its first step, we have $h^0(\Bb_i )>0$ and so $h^0(\Aa_i)\ge 2$, a contradiction. Thus each  $\Aa _i$ is indecomposable.

\begin{remark}\label{ab4}
Let $u: \Aa \to \Bb$ be a surjection of sheaves on $C$. Since $\dim C =1$, we have $h^2(C,\mathrm{ker}(u)) =0$. Thus the surjection $u$ induces a surjective map $H^1(C,\Aa )\to H^1(C,\Bb )$.
\end{remark}

\begin{lemma}\label{abb1}
Let $\Mm, \Dd _1, \Dd _2$ be vector bundles on $C$ fitting into exact sequences
\begin{equation}\label{eqabb1}
0 \to \Mm \stackrel{u_i}{\to} \Dd _i\to \Oo _C\to 0,
\end{equation}
corresponding to an extension class $e_i \in \Ext_C^1(\Oo_C, \Mm)\cong H^1(\Mm)$ for each $i$. If there exists an isomorphism $h: \Dd _2\rightarrow \Dd _1$ such that $h(u_2(\Mm)) = u_1(\Mm)$, then $e_1$ and $e_2$ are in the same orbit of $H^1(\Mm)$ for the action of the group $\mathrm{Aut}(\Mm)$.
\end{lemma}

\begin{proof}
Note that $h^0(\Mm )\le h^0(\Dd _i)\le h^0(\Mm )+1$, and $h^0(\Mm)=h^0(\Dd _i)$ if and only if $e_i\ne 0$. Since $h$ is an isomorphism, $e_1=0$ if and only if $e_2=0$. Since the assertion is obvious when $e_1=e_2=0$, we may assume $e_1\ne 0$ and $e_2\ne 0$. Since $h(u_2(\Mm)) = u_1(\Mm)$, $h$ induces isomorphisms $h': D_2/u_2(\Mm )\to \Dd _1/u_1(\Mm)$ and $f: \Mm \to \Mm$. Since $\Dd
_i/u_i(\Mm )\cong \Oo _C$, $i=1,2$, $h'$ is induced by the multiplication by a constant, $c$. Note that $e_i$ is determined by the image of $1$ by the coboundary map $H^0(\Oo _C)\to H^1(\Mm )$ in (\ref{eqabb1}). Since $e_1\ne 0$ and $e_2\ne 0$, we have $c\ne 0$. Taking $\left(\frac{1}{c}\right)h$ instead of $h$ we reduce to the case in which $h': \Oo _C\rightarrow \Oo _C$ is the identity map. Thus we get a commutative diagram with exact rows:
$$\begin{array}{ccccccc}
0\to &  \Mm      &\to  &\Dd _2&  \to  &\Oo _C &\to 0 \\
& \downarrow& &\downarrow & &\downarrow & \\
0 \to &\Mm &\to  &\Dd _1&\to &\Oo _C &\to 0,\\
\end{array}$$
in which the three vertical arrows are respectively $f$, $h$ and $\mathrm{Id}_{\Oo _C}$. By the definition of $\Ext_C^1(\Oo _C,\Mm )$ as short exact sequences modulo an equivalence relation, we get $e _1 = f_\ast (e_2)$, i.e. $e_1 \in H^1(\Mm)$ is contained in the orbit of $e_2$ for the action of the group $\mathrm{Aut}(\Mm)$.
\end{proof}

We set $\mathbf{T}_2:= H^1(\Oo _C)\setminus \{0\}$ and consider it as a parameter space, not finite-to-one, for non-trivial extensions of $\Oo _C$ by $\Oo _C$. Then we get a family $\{\Aa _2(e)\}_{e\in \mathbf{T}_2}$ of aCM vector bundles of rank two. Since we have $h^1(\Aa_2(e))=2g-3$ for each $e\in \mathbf{T}_2$, there is a vector bundle $\pi_2 : \mathbf{T}_3' \rightarrow \mathbf{T}_2$ of rank $2g-3$ whose fibre over $\Aa_2(e)$ is $H^1(\Aa_2(e)) \cong \Ext_C^1(\Oo_C, \Aa_2(e))$. Then we get a family $\{\Aa _3(e)\}_{e\in \mathbf{T}'_3}$ of aCM vector bundles of rank three on $C$ such that for each $e\in \mathbf{T} '_2$, $\Aa _3(e)$ is an extension of $\Oo _C$ by $\Aa _2(\pi (e))$. Let $\mathbf{T}_3$ be the non-empty Zariski open subset of $\mathbf{T} '_3$ parametrizing the non-trivial extensions of $\Oo _C$ by $\Aa _2(\pi (e))$. Thus we have a family $\{\Aa _3(e)\}_{e\in \mathbf{T}_3}$ of indecomposable aCM vector bundles of rank three, parametrized by $\mathbf{T} _3$.

Now we define a parameter space $\mathbf{T}_i$ inductively: fix an integer $i \ge 2$ and assume that $\mathbf{T}_i$ is defined, together with a family $\{\Aa_i (e)\}_{e\in \mathbf{T}_i}$ of indecomposable aCM vector bundles of rank $i$, parametrized by $\mathbf{T}_i$. Since we have $h^1(\Aa_i(e))=i(g-1)+1$, there exists a vector bundle $\pi_i : \mathbf{T}'_{i+1} \rightarrow \mathbf{T}_i$ of rank $i(g-1)+1$ and a family $\{\Aa _{i+1}(e)\}_{e\in \mathbf{T}'_{i+1}}$ of aCM vector bundles of rank $i+1$ on $C$ such that for each $e\in \mathbf{T} '_{i+1}$, $\Aa _{i+1}(e)$ is an extension of $\Oo _C$ by $\Aa _i(\pi (e))$. Let $\mathbf{T}_{i+1}$ be the non-empty Zariski open subset of $\mathbf{T} '_{i+1}$ parametrizing the non-trivial extensions of $\Oo _C$ by $\Aa _i(\pi (e))$.

If a vector bundle $\Aa=\Aa_r$ of rank $r$ on $C$ corresponding to $e\in \mathbf{T}_r$ is obtained as a successive extension of $\Oo_C$ by $\Aa_i(e_{i-1})$ corresponding to $e_i \in H^1(\Aa_{i}(e_{i-1}))\setminus \{0\}$ for each $i\le r$, then we simply denote it by $\Aa(e_1, \ldots, e_{r-1}):=\Aa$ and it has a filtration
$$0\subset \Aa_1=\Oo_C \subset \Aa_2=\Aa(e_1) \subset \Aa_3=\Aa(e_1, e_2)\subset \cdots \subset \Aa_r=\Aa(e_1, \dots, e_{r-1}).$$
Fix a general $\Aa = \Aa(e_1,\dots ,e_{r-1})$ that is a non-trivial extension of $\Oo _C$ by $\Aa' :=\Aa(e _1,\dots ,e_{r-2})$. Letting $u_{i,r}: \Aa _i\rightarrow \Aa$ with $1\le i\le r-1$ be the inclusion arising by the extensions reaching $\Aa$, we have the following commutative diagram
$$\begin{array}{ccccccc}
&0 & &0& & & \\
& \downarrow& &\downarrow & &&\\
 &  \Aa_1    & = &\Aa_1  &  & & \\
& \downarrow& &\downarrow & & & \\
0 \to &\Aa'&\to &\Aa &\to & \Oo_C& \to 0\\
&\downarrow& &\downarrow & &\| &\\
0 \to &\Aa'/u_{1,r-1}(\Aa_1) & \to &  \Aa/u_{1,r}(\Aa_1) &\to & \Oo_C &  \to 0\\
&\downarrow& &\downarrow & & & \\
&0 &          &0& & & \\
\end{array}$$
so that $\Aa /u_{1,r}(\Aa _1)$ is an extension of $\Oo _C$ by $\Aa' /u_{1,r}(\Aa _1)$. Iterating the process, we see that $\Aa/u_{1,r}(\Aa _1)$ is an iterated extension of $\Oo _C$.

\begin{lemma}\label{alb4}
Fix a general $\Aa_r=\Aa(e_1, \dots, e_{r-1})\in \mathbf{T}_r$ with a filtration $\Aa _1\subset \cdots \subset \Aa _{r-1}\subset \Aa _r$. Then we have
\begin{itemize}
\item [(i)] $h^0(\Aa _i/\Aa _j) =1$ for all $1\le j<i\le r$;
\item [(ii)] $f(\Aa _i)\subset \Aa _i$ for any $f\in \End (\Aa _r)$ and each $i$;
\item [(iii)] $\dim \End (\Aa _r)\le r$ and $\dim \End (\Aa _r)-\dim (\Aa _{r-1}) \le 1$.
\item [(iv)] $h(\Aa_i)=\Bb_i$ for all $i$ and any isomorphism $h: \Bb_r \rightarrow \Aa_r$, where $\Bb_r\in \mathbf{T}_r$ general with a filtration $\Bb_1 \subset \cdots \subset \Bb_{r-1} \subset \Bb_r$.
\end{itemize}
\end{lemma}

\begin{proof}
For (i) consider the following sequence, obtained from (\ref{eqalb0}):
\begin{equation}\label{equu2}
0\to \Aa _i/\Aa _j\to \Aa _{i+1}/\Aa _j\to\Oo _C\to 0.
\end{equation}
Since $e_i\in H^1(\Aa _i)$ is general by the generality of $\Aa _r$, we get that (\ref{equu2}) is a general extension and $h^0(\Aa _{i+1}/\Aa _j) =h^0(\Aa _i/\Aa _j)$. Thus to prove the assertion for $j=1$ it is enough to show it for the case $i=2$, which is obvious from $\Aa _2/\Aa _1\cong \Oo _C$. For $j\ge 2$ we use (\ref{equu2}) starting from the case $i=j+1$, when we have $\Aa _{j+1}/\Aa _j\cong \Oo _C$.

For (ii) note first that $\Aa _1=\Oo _C$ and $h^0(\Aa _r)=1$. This implies that $\Aa _1$ is the image of the evaluation map $H^0(\Aa _r)\otimes \Oo _C\rightarrow \Aa _r$ and so $f(\Aa _1)\subseteq \Aa _1$, concluding the case $r=2$. Now $f$ induces a map $f': \Aa _r/\Aa _1\rightarrow \Aa _r/\Aa_1$. Since $h^0(\Aa _r/\Aa _1)=1$ by (i) and $\Aa _2/\Aa _1\cong \Oo _C$, we get $f'(\Aa _2/\Aa _1)\subseteq \Aa _2/\Aa _1$ and so $f(\Aa _2)\subseteq \Aa _2$. Thus we get the assertion by continuing this process together with (i).

For (iii) since the case $r=1$ is trivial, we may assume $r\ge 2$ and use induction on $r$. For $f\in \End (\Aa _r)$, we have $\Aa_1 = \Oo_C$ and $f(\Aa _1) \subseteq \Aa _1$ by (ii). Thus there is $c\in \mathbf{k}$ such that $(f-c{\cdot}\mathrm{Id}_{\Aa _r})(\Aa _1)=0$, and $f-c{\cdot}\mathrm{Id}_{\Aa _r}$ is uniquely determined by $f'\in \End (\Aa _r/\Aa _1)$. Since we may apply (i) and (ii) to $\Aa _r/\Aa _1$, we conclude by induction on $r$.

For (iv) note that $\Aa _1$ (resp. $\Bb_1)$ is the image of the evaluation map of $\Aa_r$ (resp. $\Bb _r$) and $h$ is an isomorphism. In particular, we have $h(\Aa _1)=\Bb_1$ and so $h$ induces an isomorphism $h': \Aa _r/\Aa _1\rightarrow \Bb _r/\Bb _1$. Since $h^0(\Aa _i/\Aa _j) =h^0(\Bb _i/\Bb _j) =1$ for all $i>j$ by (i), we iterate the previous argument.
\end{proof}

Define a subset $\mathbf{J}_r$ to be
\begin{align*}
\mathbf{J}_r=\Bigg\{ e\in \mathbf{T}_r~\Bigg|~ \makecell{\Aa_r(e) \text{ admits a filtration }\Aa_1 \subset \cdots \subset \Aa_{r-1} \subset \Aa_r \\ \text{ such that }h^0(\Aa_i/\Aa_j)=1 \text{ for all }1\le j<i\le r}\Bigg\},
\end{align*}
i.e. the non-empty open subset of $\mathbf{T}_r$ parametrizing the vector bundles $\Aa _r$ satisfying (i) of Lemma \ref{alb4}; thus $\Aa_r$ satisfies (ii), (iii) and (iv) of Lemma \ref{alb4}.

\begin{lemma}\label{alb8}
For a general $\Aa _r\in \mathbf{J}_r$ there exists an algebraic subset of $\mathbf{J}_r$, parametrizing the  vector bundles isomorphic to $\Aa _r$, with dimension at most $\frac{r(r-1)}{2}$.
\end{lemma}

\begin{proof}
We use induction on $r$; the case $r=1$ is trivial, because $\mathbf{J} _1=\mathbf{T}_1= \{\Oo _C\}$. We assume that $r\ge2$ and fix $\Bb _r\in \mathbf{J}_r$, isomorphic to $\Aa _r$, with a filtration $\Bb _1\subset \cdots \subset \Bb _r$. For any isomorphism $h: \Bb _r\rightarrow \Aa _r$, we have $h(\Bb _{r-1}) = \Aa _{r-1}$ by (iv) of Lemma \ref{alb4}. Since $\Aa _{r-1}$ is also general in $\mathbf{J}_{r-1}$, by inductive assumption there is an algebraic subset $\mathbf{J}'$ of $\mathbf{J}_{r-1}$ parametrizing the vector bundles isomorphic to $\Aa _{r-1}$. Fix $\Mm \in \mathbf{J}'$ and consider the subset $\mathbf{T}'  \subset \mathbf{T}_r$ of all extensions of $\Oo _C$ by $\Mm$ which are isomorphic to $\Aa _r$. By Lemma \ref{abb1} and (iii) of Lemma \ref{alb4}, we have $\dim \mathbf{T}' \le r-1$ and we get the assertion.
\end{proof}

\begin{proof}[Proof of Proposition \ref{alb1}:]
Note that
\[
 g-1+\sum _{i=2}^{r-1} \left(i(g-1)-1\right) -\sum _{i=1}^{r-1} i=\frac{(r-1)(r-2)(g-1)}{2} -\frac{r(r-1)}{2}.
\]
Set $\Delta:=\{v ^\ast (\Aa)~|~\Aa\in \mathbf{J}_r\}$ and then each element of $\Delta$ is indecomposable, because each $\Aa \in \mathbf{J}_r$ is indecomposable. Since we have $v_*v^*\Ff \cong \Ff$ for any vector bundle $\Ff$ on $C$ by the projection formula and $v_*\Oo_X\cong \Oo_C$, we have $v^*\Aa \cong v^*\Bb$ if and only if $\Aa \cong \Bb$ for any $v^*\Aa, v^*\Bb \in \Delta$.

Fix a general $\Ll\in \mathrm{Pic}^0(X)$ and set $\Theta _{\Ll}:=\{\Gg \otimes \Ll~|~\Gg\in \Delta\}$. Each element of $\Theta _{\Ll}$ is an indecomposable vector bundle of rank $r$ on $X$ and the isomorphism classes of elements in $\Theta _{\Ll}$ are also parametrized by $\mathbf{J}_r$. We have $h^1(\Ll) =0$ by \cite[Th. 0.1]{beau}, because $q(X)>g$ and by our definition of $g$ there is no non-constant morphism from $X$ to a curve of genus $q(X)$. Fix a positive interger $t$. By Kleiman's numerical criterion of ampleness in \cite{k}, $\Ll^\vee (t)$ and $\omega _X^\vee \otimes \Ll(t)$ are ample. So Kodaira' vanishing gives $h^1(\Ll(-t)) = h^1(\omega_X\otimes \Ll^\vee (-t)) =0$. On the other hand, by Serre's duality we get $h^1(\Ll(t)) = h^1(\omega _X\otimes \Ll^\vee (-t)) =0$. Thus $\Ll$ is aCM.

Since each element of $\Theta _{\Ll}$ is an iterated extension of $\Ll$, each element of $\Theta _{\Ll}$ is also aCM. Note that each element of $\Theta _{\Ll}$ is strictly semistable with $gr (\Aa _r)\cong \Ll^{\oplus r}$ and so no element of $\Theta _{\Ll}$ is isomorphic to an element of $\Theta_{\Ll'}$ with $\Ll\ncong \Ll'$. Now we may vary the general $\Ll \in \mathrm{Pic}^0(X)$ to obtain a family $\Gamma$ whose fibre over $\Ll$ is $\Theta_{\Ll}$. Then we get the inequality in the assertion and all the requirements for $\Gamma$ are satisfied.
\end{proof}


\bibliographystyle{amsplain}
\providecommand{\bysame}{\leavevmode\hbox to3em{\hrulefill}\thinspace}
\providecommand{\MR}{\relax\ifhmode\unskip\space\fi MR }
\providecommand{\MRhref}[2]{%
  \href{http://www.ams.org/mathscinet-getitem?mr=#1}{#2}
}
\providecommand{\href}[2]{#2}

\end{document}